\documentclass[11pt, reqno, a4paper]{amsart}
\usepackage[margin=1.2in]{geometry}
\numberwithin{equation}{section}
\usepackage{amssymb,amsfonts,amsthm}
\usepackage[utf8]{inputenc}
\usepackage{listings}
\usepackage{bm}
\usepackage[hyperpageref]{backref}
 \usepackage{tikz}
\usetikzlibrary{arrows.meta,arrows}
\usepackage[shortlabels]{enumitem}
\usepackage{epsfig}
\usepackage{color}
\usepackage[bookmarks]{hyperref}

\allowdisplaybreaks[4]

\newtheoremstyle{myremark}{10pt}{10pt}{}{}{\bfseries}{.}{.5em}{}

 \newtheorem{thm}{Theorem}[section]
 
 \newtheorem{lemma}[thm]{Lemma}
 \newtheorem{prop}[thm]{Proposition}
 \theoremstyle{definition}
 \newtheorem{defn}[thm]{Definition}
 \newtheorem{example}[thm]{Example}
 \newtheorem{remark}[thm]{Remark}
\newtheorem{theorem}[thm]{Theorem}

 \newcommand{\D}{\bigtriangleup_2}

\allowdisplaybreaks[4]

\title[On fractional Orlicz-Hardy inequalities]{On fractional Orlicz-Hardy inequalities}

\author[T. V. Anoop, P. Roy, and S. Roy]{T. V. Anoop$^{1,*}$, Prosenjit Roy$^{2}$, and Subhajit Roy$^{3}$}



{\email{anoop@iitm.ac.in, prosenjit@iitk.ac.in, rsubhajit.math@gmail.com}}

\thanks{$^*$Corresponding author.}
\begin{document}

\maketitle

\centerline{$^{1,3}$Department of Mathematics, Indian Institute of Technology Madras,
 }
\centerline{Chennai  600036, India}
\centerline{$^{2}$Department of Mathematics and Statistics, Indian Institute of Technology Kanpur,}
\centerline{Kanpur 208016, India}

\begin{abstract} We establish the weighted fractional Orlicz-Hardy inequalities for various  Orlicz functions. Further, we identify the critical cases for each Orlicz function and prove the weighted fractional Orlicz-Hardy inequalities with logarithmic correction. Moreover, we discuss the analogous results in the local case. In the process, for any Orlicz function $\Phi$  and for any $\Lambda>1$, the following inequality is established
$$
 \Phi(a+b)\leq \lambda\Phi(a)+\frac{C(
 \Phi, \Lambda
 )}{(\lambda-1)^{p_\Phi^+-1}}\Phi(b),\;\;\;\forall\,a,b\in [0,\infty),\,\forall\,\lambda\in (1,\Lambda],
$$
 where $p_\Phi^+:=\sup\big\{t\varphi(t)/\Phi(t):t>0\big\},$ $\varphi$ is the right derivatives of $\Phi$ and $C(
 \Phi, \Lambda
 )$ is a positive constant that depends only on $\Phi$ and $\Lambda.$
\end{abstract}

\section{Introduction}
For  $N\in \mathbb{N}$, recall the classical {\it{Hardy inequality}}:  
\begin{equation}\label{p-Hardy}
    \int_{\mathbb{R}^N}\frac{|u(x)|^p}{|x|^p}dx\leq \left|\frac{p}{N-p}\right|^p\int_{\mathbb{R}^N}|\nabla u(x)|^pdx,
\end{equation}  
for all $ u\in \mathcal{C}_c^1(\mathbb{R}^N)$ if $1<p<N$ and for all $ u\in \mathcal{C}_c^1(\mathbb{R}^N\setminus\{0\})$ if $p>N$ (see \cite[Theorem 1.2.5]{Evans2015}).
The above inequality has been extended in several directions. For example, in  \cite{Adimurthi2002, Anoop2021, Anoop2015, Tarantello2002}, the Hardy potential $\frac{1}{|x|^p}$ is replaced with more general weight functions and   \cite{subha,  Mihai2012, Cianchi1999, Kal2009} replaced the convex function $t^p$ with a more general \textit{Orlicz function} satisfying certain sufficient conditions.
Another extension of Hardy's inequalities is the  {\it{Caffarelli–Kohn–Nirenberg (C-K-N) inequalities}}, which were established by Caffarelli, Kohn, and Nirenberg in \cite{C-K-N1982,C-K-N1984}.
In \cite{Nguyen2018}, Nguyen and Squassina established the fractional version of the C-K-N inequalities. We first introduce some notations to state a particular case of their result. For $u\in \mathcal{C}(\mathbb{R}^N)$ and $s\in (0,1)$, let $D_su$ be the s-H\"older quotient and $d\mu$ be the product measure on $\mathbb{R}^N\times \mathbb{R}^N$ defined as $$D_su(x,y)=\frac{u(x)-u(y)}{|x-y|^s},\;\;\;d\mu=\frac{dxdy}{|x-y|^N}.$$
Let $p>1$ and $\alpha_1,\,\alpha_2,\,\,\gamma\in \mathbb{R}$ be such that $\gamma=s-\alpha_1-\alpha_2.$    
 Then Theorem 1.1 of \cite{Nguyen2018} (with $\tau=p,\,a=1$) establishes the  following  {\it{weighted fractional Hardy inequalities}:}
 \begin{enumerate}[(i)]
     \item  for $\gamma<N/p,$ 
 \begin{equation}\label{frac-norm}
   \int_{\mathbb{R}^N}\left(\frac{|u(x)|}{|x|^{\gamma }}\right)^p dx\leq C \int_{\mathbb{R}^N}\int_{\mathbb{R}^N}\Big(|x|^{\alpha_1}|y|^{\alpha_2 }|D_su(x,y)|\Big)^p d\mu,\;\;\;\,\forall\,u\in\mathcal{C}_c^1(\mathbb{R}^N),
\end{equation}
\item for $\gamma>N/p,$ 
\begin{equation}\label{frac-norm*}
   \int_{\mathbb{R}^N}\left(\frac{|u(x)|}{|x|^{\gamma }}\right)^p dx\leq C \int_{\mathbb{R}^N}\int_{\mathbb{R}^N}\Big(|x|^{\alpha_1}|y|^{\alpha_2 }|D_su(x,y)|\Big)^p d\mu,\;\;\;\,\forall\,u\in\mathcal{C}_c^1(\mathbb{R}^N\setminus\{0\}),
   \end{equation}
 \end{enumerate}
 where $C$ is a positive constant independent of $u.$ The above inequalities in the case $\alpha_1=\alpha_2=0$ with the best constants were obtained for $p=2$ in \cite{herbst1977} and for any $p\geq 1$ in \cite{frank2008}. We refer to \cite{Beckner2008, Leib2008, Mazya2002} for further reading on these inequalities.  
For $\alpha_1=\alpha_2\in\big((sp-N)/2p,0\big]$,  \eqref{frac-norm}  was proved in 
\cite{abdell2017}. In \cite{Dyda2022}, the above inequalities were derived  for $\alpha_1+\alpha_2\in (-N/p,s)$. The Hardy inequality has also been extended to general domains (known as boundary Hardy inequality) by replacing $|x|$ with the distance function from the boundary of the underlying domain; see \cite{adisahuroy, dyda2011, Dyda2007, Li2012, Loss2010} and the references therein. For some recent developments in Hardy inequalities; see \cite{weiwei2022, Chen2021, Li2023}. We are interested in generalizing \eqref{frac-norm} and \eqref{frac-norm*} by replacing the convex function $t^p$ with a more general Orlicz function.

 \begin{defn}[Orlicz function]
    A  continuous, convex function  $\Phi:[0,\infty)\to [0,\infty)$ is called an Orlicz function if it has the following properties:
    \begin{enumerate}[(a)]
        \item $\lim_{t\to 0}\frac{\Phi(t)}{t}=0$ and $\lim_{t\to \infty}\frac{\Phi(t)}{t}=\infty$,
        \item $\Phi$ satisfies the  $\D$-condition, i.e. there exists a constant $C> 0$ such that 
\begin{equation*}
    \Phi(2t)\leq C\Phi(t),\;\;\;\forall\,t\geq 0.
\end{equation*}
    \end{enumerate}
 \end{defn}
\noindent It follows from \cite[Theorem 1.1]{Krasn1961} that an Orlicz function can be represented in the form
\begin{equation}\label{right-deri}
    \Phi(t)=\int_0^t \varphi(s) ds\;\;\;  \text{for}\,\, t\geq 0,
\end{equation} 
where $\varphi$ is a non-decreasing right continuous function on $[0,\infty)$ satisfying $\varphi(0)=0,$ $\varphi(t)>0$ for $t>0,$ and $\lim_{t\to \infty}\varphi(t)=\infty$.
\smallskip

Associated to an Orlicz function $\Phi$, we define $p_\Phi^-$ and $p_\Phi^+$ (cf. \cite{Mihai2012}) as 
 \begin{equation*}
p_\Phi^-:=\inf_{t>0}\frac{t\varphi(t)}{\Phi(t)},\;\;\;\;\;\; p_\Phi^+:=\sup_{t>0}\frac{t\varphi(t)}{\Phi(t)}.
 \end{equation*}
Notice that, for $\Phi(t)=t^p$ with $p\in (1,\infty)$ we have $p_\Phi^+=p_\Phi^-=p.$
 For an Orlicz function $\Phi$, it is easy to see that  $$\Phi(t)\leq t\varphi(t)\leq\Phi(2t)\leq C \Phi(t),\;\;\; \forall\,t \in[0,\infty).$$ This implies that $ 1\le p_\Phi^- \leq p_\Phi^+< \infty.$ 
\smallskip

We say two Orlicz functions $\Phi$ and   $\Psi$ are equivalent  ($\Phi\asymp\Psi$) if there exist $C_1,C_2>0$ such that $$C_1\Phi(t)\le \Psi(t)\le C_2\Phi(t),\quad\forall\,t\ge 0.$$
If $\Phi\asymp\Psi,$ then one can verify that $p_\Psi^- \le p_\Phi^+$ and $p_\Phi^- \le p_\Psi^+$ (see Lemma \ref{lemm-equivalent}). 
Now, for a given Orlicz function $\Phi$, we can define the following two quantities :
  \begin{align*}
      p_\Phi^\ominus&= \sup\{p_\Psi^-: \Phi \asymp \Psi\}, \\
      p_\Phi^\oplus&= \inf\{p_\Psi^+: \Phi \asymp \Psi\}.      
  \end{align*}
 Therefore,  $ p_\Phi^- \le p_\Phi^\ominus \le  p_\Phi^\oplus \le  p_\Phi^+.$ Indeed, $p_\Phi^\ominus$ and  $p_\Phi^\oplus$ remain the same for all equivalent Orlicz functions. In particular, for $\Phi \asymp A_p,$ where $A_p(t):=t^p$ we have  $p_\Phi^\ominus = p_\Phi^\oplus=p.$ 

\smallskip

In this article, depending on the Orlicz function $\Phi$ and  values of $s\in (0,1),\,\alpha_1,\alpha_2\in \mathbb{R}$,  we identify some range of $\gamma=s-\alpha_1-\alpha_2$ for which the following \textit{weighted fractional Orlicz-Hardy inequality} holds:
\begin{equation}\label{Hardy}
    \int_{\mathbb{R}^N}\Phi\left(\frac{|u(x)|}{|x|^\gamma}\right)dx\leq C\int_{\mathbb{R}^N}\int_{\mathbb{R}^N}\Phi\left(|x|^{\alpha_1}|y|^{\alpha_2}|D_su(x,y)|\right)d\mu,\;\;\;\forall\,u\in \mathcal{C}_c^1(\mathbb{R}^N),
\end{equation}
 where $C$ is a positive constant. 
 For $\alpha_1=\alpha_2=0$ and $\gamma=s<N/p^+_\Phi$, the above inequality is established in \cite[Theorem 1.2]{salort2022*}. The above inequality with $\alpha_1=\alpha_2=0$ and $\gamma=s<N/p^\oplus_\Phi$ can  be derived from \cite[Theorem 5.1 and Proposition C]{Cianchi2022}.
The following is our first main result:
\begin{theorem}\label{Orlicz-Hardy} Let $N\geq 1,\,s\in (0,1)$, $\alpha_1, \, \alpha_2\in \mathbb{R}$,  and let
$\gamma:=s-\alpha_1-\alpha_2.$ For an Orlicz function $\Phi$, if  $\gamma<N/p^\oplus_\Phi$,  then \eqref{Hardy} holds and if $\gamma> N/p^\ominus_\Phi$, then \eqref{Hardy} fails. Furthermore, if $p^\ominus_\Phi$ is attained, then \eqref{Hardy}
fails also for $\gamma=N/p^\ominus_\Phi$.
\end{theorem}

Our proof is based on the dyadic decomposition of  $\mathbb{R}^N$, Poincar\'e inequalities for annulus, and a clever summation process to pass the information from a family of annuli to the whole space. The similar ideas are used in \cite{Nguyen2018} and \cite{Nguyenmarco2019}.\smallskip

If we restrict the class of functions to $\mathcal{C}_c^1(\mathbb{R}^N\setminus 
\{0\})$, then  we can obtain an analogue of \eqref{frac-norm*} for $\gamma>N/p^\ominus_\Phi$. Namely, we get the following weighted fractional Orlicz-Hardy inequality:
\begin{equation}\label{Hardy-0}
\int_{\mathbb{R}^N}\Phi\left(\frac{|u(x)|}{|x|^\gamma}\right)dx\leq C\int_{\mathbb{R}^N}\int_{\mathbb{R}^N}\Phi\left(|x|^{\alpha_1}|y|^{\alpha_2}|D_su(x,y)|\right)d\mu,\;\;\;\forall\,u\in \mathcal{C}_c^1(\mathbb{R}^N\setminus\{0\}),
\end{equation}
where $C$ is a positive constant. Notice that the above inequality holds for $\gamma<N/p^\oplus_\Phi$ (see Theorem \ref{Orlicz-Hardy}).
For $\alpha_1=\alpha_2=0$, $N=1$, and $\gamma=s>1/p^-_\Phi$, the above inequality was established in  \cite[Theorem 1.1]{salort2022}. For any  $N\ge 1,$  \eqref{Hardy-0}  with $\alpha_1=\alpha_2=0$ and $\gamma=s>N/p^\ominus_\Phi$, can be derived from \cite[Theorem 1.4]{roy2022}.
The following result allows $\alpha_1$ and $\alpha_2$ to be any real numbers. 
 \begin{theorem}\label{Orlicz-Hardy2} Let $N\geq 1,\,s\in (0,1)$, $\alpha_1, \, \alpha_2\in \mathbb{R}$,  and let
$\gamma:=s-\alpha_1-\alpha_2.$ Then, for an Orlicz function $\Phi$,  \eqref{Hardy-0} holds also  if $\gamma>N/p^\ominus_\Phi$.
\end{theorem}
\smallskip

Now we consider the limiting cases: $\gamma=N/p^\oplus_\Phi$ and $\gamma=N/p^\ominus_\Phi$. It is known that  for ~$ \Phi(t)=t^p$;  both \eqref{Hardy} and \eqref{Hardy-0} (with $\alpha_1=\alpha_2=0$)  fails to hold for $\gamma=N/p$ (see \cite[Page 578]{Dyda2007}). Indeed, for any Orlicz function $\Phi$, if $p_\Phi^\ominus$ is attained, then \eqref{Hardy} fails to hold for $\gamma=N/p^\ominus_\Phi$ (see Theorem \ref{Orlicz-Hardy}). 
In the case of ~$S(t)=t^p+t^q$, \eqref{Hardy} fails to hold also  for $\gamma=N/p^+_S=N/p_S^\oplus$ (see Remark \ref{rem2}).
  Recall that, even \eqref{p-Hardy} (the classical Hardy inequality) does not hold for $p=N$. In  \cite[Page 49]{Leray1933},  Leray observed that   $\frac{1}{|x|^2(\log(e/|x|))^2}$  is the right Hardy potential to have an analogue of  \eqref{p-Hardy} for $p=N=2$.  More precisely,  
Leray established the following inequality: 
\begin{equation*}
    \int_{B_1(0)}\frac{|u(x)|^2}{|x|^2\log^2(e/|x|)}dx\leq C\int_{B_1(0)}|\nabla u(x)|^2dx,\;\;\;\forall\, u\in \mathcal{C}_c^1(B_1(0)),
\end{equation*}
where $B_1(0)\subset \mathbb{R}^2$ is the ball centred at the origin with radius 1. See also \cite[Theorem 1.1]{Adimurthi2002} for similar inequalities on bounded domains in $\mathbb{R}^N.$ 
For more general Hardy-type potentials in the critical case; see \cite[Theorem 1]{Anoop2012} for bounded domains in $\mathbb{R}^2$ and \cite[Lemma 2.2]{Anoop2015} for exterior domains in $\mathbb{R}^N.$ 

\smallskip
   
The next theorem considers the cases $\gamma=N/p^\oplus_\Phi$ and $\gamma=N/p^\ominus_\Phi$ with logarithmic correction.
\begin{theorem}\label{log-orlicz}
    Let $N\geq 1,\,s\in (0,1)$, $\alpha_1, \alpha_2\in \mathbb{R}$, and let $\gamma:=s-\alpha_1-\alpha_2$.   
For an Orlicz function $\Phi$, 
\begin{enumerate}[(i)]
    \item  if $p^\oplus_\Phi$ is attained  and $\gamma=N/p^\oplus_\Phi$, then for every $u\in \mathcal{C}_c^1(\mathbb{R}^N)$ with $supp(u)\subset B_R(0),$ 
\begin{equation}\label{one}
           \int_{B_R(0)}\frac{\Phi(|x|^{-\gamma}|u(x)|)}{\big(\log(2R/|x|)\big )^{p^\oplus_\Phi}}dx\leq C\int_{\mathbb{R}^N}\int_{\mathbb{R}^N}\Phi\left(|x|^{\alpha_1}|y|^{\alpha_2}|D_su(x,y)|\right)d\mu,
        \end{equation}
        \item if $p^\ominus_\Phi$ is attained at $\Psi\asymp\Phi$ and $\gamma=N/p^\ominus_\Phi$, then for every $u\in \mathcal{C}_c^1(\mathbb{R}^N)$ with $supp(u)\subset B_R(0)^c,$
 \begin{equation}\label{two}
           \int_{B_R(0)^c}\frac{\Phi(|x|^{-\gamma}|u(x)|)}{\big(\log(2|x|/R)\big )^{p^+_\Psi}}dx\leq C\int_{\mathbb{R}^N}\int_{\mathbb{R}^N}\Phi\left(|x|^{\alpha_1}|y|^{\alpha_2}|D_su(x,y)|\right)d\mu,
        \end{equation} 
\end{enumerate}
where $B_R(0)$ is the open ball centred at the origin with radius $R>0$ and $C$ is a positive constant independent of $u$.
\end{theorem}

\smallskip

For $\Phi(t)=t^p,\,p>1$ and $\alpha_1=\alpha_2=0,$  Edmunds and Triebel established the above theorem in \cite[Theorem 2.8]{Hans1999}  by using interpolation techniques.  In \cite[Theorem 3.1]{Nguyen2018}, using the dyadic decomposition of  $\mathbb{R}^N$, Nguyen and Squassina established the above theorem for the special case $\Phi(t)=t^p$. In fact, they have established a full-range fractional version of C-K-N inequalities.

\begin{remark}
By Theorem \ref{Orlicz-Hardy}, for $\gamma<N/p^\oplus_\Phi$,  \eqref{one} holds with any bounded function in place of $1/\big(\log(2R/|x|)\big)^{p_\Phi^\oplus}$. 
Similarly, by Theorem \ref{Orlicz-Hardy2}, for  $\gamma>N/p^\ominus_\Phi$, \eqref{two} holds with any bounded function in place of $1/\big(\log(2|x|/R)\big)^{p_\Psi^+}$. 
\end{remark}

\begin{remark}  
For any $\Phi \asymp A_p,$ $p_\Phi^\ominus$ and $p_\Phi^\oplus$ are attained at $A_p$. See Example \ref{exm-o} for more examples of Orlicz functions for which these two quantities are attained. We do not know whether these quantities are always attained for an Orlicz function or not. 
If  $p^\oplus_\Phi$ is not attained, then the question `whether \eqref{one} holds for $\gamma=N/p^\oplus_\Phi$ ?' is open. Similarly,  if $p^\ominus_\Phi$ is not attained, then for $\gamma=N/p^\ominus_\Phi$, whether \eqref{two}  holds with some power of $\log(2|x|/R)$ remains an open question.
\end{remark}

\begin{remark} For proving Theorem 3.1 of \cite{Nguyen2018}, Lemma 2.2 of \cite{Nguyen2018}  plays an important role. This lemma states that for $p,\Lambda\in(1,\infty),$ there exists $C=C(p,\Lambda)>0$ so that 
     \begin{equation*}
(a+b)^p\leq \lambda a^p+\frac{C}{(\lambda-1)^{p-1}}b^p,\;\;\;\forall\,a,b\in [0,\infty),\,\forall\,\lambda\in (1,\Lambda). 
  \end{equation*} 
This article proves an analogue of the above inequality for any Orlicz function (see Lemma \ref{Lem-H}). More precisely, for any $\Lambda>1,$ we prove the existence of $C=C(\Phi,\Lambda)>0$ satisfying the following inequality:
     \begin{equation*}
 \Phi(a+b)\leq \lambda\Phi(a)+\frac{C}{(\lambda-1)^{p^+_\Phi-1}}\Phi(b),\;\;\;\forall\,a,b\in [ 0,\infty),\,\forall\,\lambda\in (1,\Lambda].  
  \end{equation*} 
The proof Lemma 2.2 of \cite{Nguyen2018} is based on the homogeneity of $t^p$. On the other hand, we use some subtle properties of the Orlicz function to prove the above inequality.    
\end{remark}

 The rest of this article is organized in the following way: In section 2, we recall some properties of the Orlicz function and prove some vital lemmas used in the subsequent section. We present the proofs for Theorem \ref{Orlicz-Hardy}, Theorem \ref{Orlicz-Hardy2}, and Theorem \ref{log-orlicz} in Section 3. Sections 4 and 5 contain the local analogue of Theorem 1.2-Theorem 1.4 and discuss some important results on weighted fractional Orlicz-Hardy inequalities.
 
\section{preliminaries}
In this section, we recall or prove some essential results we need to establish the main theorems of this article. Throughout this article, we shall use the following notations: 
\begin{itemize}
    \item  $\mathcal{C}_c^1(\Omega)$ denotes the set of continuously differentiable functions with compact support.
    \item For a measurable set $\Omega\subset\mathbb{R}^N$ and $u \in L^1(\mathbb{R}^N)$, $(u)_\Omega$ will denote the average of the function $u$ over $\Omega,$ i.e.,
\begin{equation*}
(u)_\Omega:=\frac{1}{|\Omega|}\int_\Omega u\,dx, 
\end{equation*}
where $|\Omega|$ is the Lebesgue measure of $\Omega.$ 
    \item  For any $f,g:\Omega\subset\mathbb{R}\rightarrow \mathbb{R},$ we
denote $f\asymp g$ if there exist positive constants  $C_1,C_2$ such that $C_1f(x)\leq g(x)\leq C_2f(x)$ for all $x\in \Omega.$ 

\end{itemize}

\subsection{Properties of an Orlicz function}
In the following lemma, we enlist some useful
inequalities involving Orlicz functions. 
\begin{lemma}
    Let $\Phi$ be an Orlicz function, and $\varphi$  be the right derivatives of $\Phi$ as given in \eqref{right-deri}. Then the following hold for every $a, b \geq 0:$
\begin{align}
\Phi(a)&\asymp a\varphi(a),\label{H1}\\   \Phi(a+b)&\leq 2^{p^+_\Phi}\left(\Phi(a)+\Phi(b)\right),\label{H2}\\
     \min\big\{a^{p^-_\Phi},a^{p^+_\Phi}\big\}\Phi(b)&\leq \Phi(ab)\leq \max\big\{a^{p^-_\Phi},a^{p^+_\Phi}\big\}\Phi(b),\label{H3}\\
     \min\big\{a^{1/p^-_\Phi},a^{1/p^+_\Phi}\big\}\Phi^{-1}(b)&\leq  \Phi^{-1}(ab)\leq  \max\big\{a^{1/p^-_\Phi},a^{1/p^+_\Phi}\big\}\Phi^{-1}(b),\label{H4}\\ C_1\min\big\{a^{p^-_\Phi-1},a^{p^+_\Phi-1}\big\}\varphi(b)&\leq  \varphi(ab)\leq C_2\max\big\{a^{p^-_\Phi-1},a^{p^+_\Phi-1}\big\}\varphi(b),\label{H6}
\end{align}
where $C_1$ and $C_2$ are positive constants depending only on $\Phi.$
\begin{proof}
For proofs of \eqref{H1}, \eqref{H2}, \eqref{H3}, and \eqref{H4}; see  \cite[Lemma 2.1]{salort2022} and \cite[Proposition 2.2]{subha}. The inequality \eqref{H6} follows from \eqref{H1} and 
 \eqref{H3}.
\end{proof}
\end{lemma}

Next, we have the following lemma for two equivalent Orlicz functions.
    \begin{lemma}\label{lemm-equivalent}
        Let $\Phi$ and $\Psi$ be two equivalent Orlicz functions ($\Phi\asymp\Psi$). Then 
        $$p_\Phi^- \le p_\Psi^+, \; \text{ and }\; p_\Psi^- \le p_\Phi^+. $$
    \end{lemma}
    \begin{proof}
     Since $\Phi\asymp\Psi$, there exists positive constant $C$ such that
     \begin{equation*}
         \frac{\Phi(t)}{\Psi(t)} \le C, \;\;\; \forall\; t> 0.
     \end{equation*}
                By  \eqref{H3}, we have
                $$t^{p^-_\Phi}\Phi(1)\le \Phi(t) \text{ and }  \Psi(t) \le t^{p^+_\Psi}\Psi(1),\;\; \forall\, t>1.$$
         From  the above inequalities, we obtain
         $$C\ge \frac{\Phi(t)}{\Psi(t)}\ge \frac{\Phi(1)t^{p^-_\Phi}}{\Psi(1)t^{p^+_\Psi}},\;\;\; \forall\, t>1.$$
         Therefore, we must have  $p_\Phi^- \le p_\Psi^+.$ By interchanging the roles of $\Phi$ and $\Psi$ we get the second inequality.
  \end{proof}

\subsection{Some function spaces:}  Let $\Omega$ be an open set in $\mathbb{R}^N,$ and $\Phi$ be an Orlicz function. 
\begin{enumerate}[(i)]
\item  \textbf{\underline{Orlicz spaces}:}
The  Orlicz space associated with $\Phi$ is  defined as 
$$ L^{\Phi}(\Omega)=\left\{u:\Omega\rightarrow \mathbb{R}\; \text{measurable}\; :\int_{\Omega}\Phi\left(|u(x)|\right)dx<\infty \right\}.$$ 
 The space $L^{\Phi}(\Omega)$  is a Banach space with respect to the following \textit{Luxemburg norm:}
\begin{equation*}
   \|u\|_{L^{\Phi}(\Omega)}=\inf\left\{ \lambda>0:\int_{\Omega}\Phi\left( \frac{|u(x)|}{\lambda}\right) dx\leq 1\right\}.
\end{equation*}

    \item  \textbf{\underline{Orlicz-Sobolev spaces}:} The Orlicz-Sobolev space is defined as
$$ W^{1,\Phi}(\Omega)=\left\{u\in L^\Phi(\Omega) :\;\;|\nabla u|\in L^\Phi(\Omega) \right\},$$
where the partial derivatives are understood in the distributional sense. The space $ W^{1,\Phi}(\Omega)$ is a Banach space with the norm $\|u\|_{W^{1,\Phi}(\Omega)}:=\|u\|_{L^\Phi(\Omega)}+\|\nabla u\|_{L^\Phi(\Omega)}$.
\smallskip

\item  \textbf{\underline{Frcational Orlicz-Sobolev spaces}:} Let $s\in (0,1)$. The fractional Orlicz-Sobolev space is defined as
$$ W^{s,\Phi}(\Omega)=\left\{u\in L^\Phi(\Omega) :\; I_{\Phi,\Omega}(u)<\infty  \right\}, \;\;\; I_{\Phi,\Omega}(u):=\int_{\Omega}\int_{\Omega}\Phi\left( |D_su(x,y)|\right)d\mu .$$
The space $ W^{s,\Phi}(\Omega)$ is a Banach space with the norm $\|u\|_{W^{s,\Phi}(\Omega)}:=\|u\|_{L^\Phi(\Omega)}+[u]_{W^{s,\Phi}(\Omega)}$, where
\begin{equation*}
    [u]_{W^{s,\Phi}(\Omega)}=\inf\left\{ \lambda>0: I_{\Phi,\Omega}\left(\frac{u}{\lambda}\right)\leq 1\right\}.
\end{equation*}
\end{enumerate}
For the details on Orlicz-Sobolev and fractional Orlicz-Sobolev spaces; see \cite{Adams1975,bonder2019, Krasn1961}.
For other related works on fractional Orlicz-Sobolev
spaces, we refer to \cite{Bahrouni2020,Rita-Peter2020,salort-jde}.
\\

Next, for a given bounded domain $\Omega$ and $\lambda>0$, we establish the {\it{fractional Poincar\'e-Wirtinger inequality}} for an Orlicz function on $\Omega_\lambda:=\{\lambda x:x\in \Omega\}$ with a   
constant which is independent of $\lambda$.
\smallskip

\begin{prop}(\textbf{Fractional Orlicz-Poincar\'e-Wirtinger inequality}.)\label{lem-orlicz}
 Let $\Omega$ be a bounded open subset of $\mathbb{R}^N,\,N\ge 1$, and   $s\in (0,1),\,\lambda>0$. Then for any Orlicz function $\Phi,$ there exists a positive constant $C=C(s,N,\Omega,\Phi)$  such that
\begin{equation*}
  \int_{\Omega_\lambda}\Phi(|u(x)-(u)_{\Omega_\lambda}|)dx\leq C\int_{\Omega_\lambda}\int_{\Omega_\lambda}\Phi\left(\lambda^s|D_su(x,y)|\right)d\mu,\;\;\;\forall\,u\in W^{s,\Phi}(\Omega_\lambda).
\end{equation*}
\end{prop}

\begin{proof}
Let $\text{d}=\text{diam} (\Omega)$ and $u\in W^{s,\Phi}(\Omega_\lambda)$. Then $\text{diam}(\Omega_\lambda)=\lambda\text{d}.$ Thus, by Jenson's inequality, 
  we obtain
    \begin{align*}
        \int_{\Omega_\lambda}\Phi(|u(x)-(u)_{\Omega_\lambda}|)dx&=\int_{\Omega_\lambda}\Phi\left(\left|\frac{1}{|\Omega_\lambda|}\int_{\Omega_\lambda}(u(x)-u(y))dy\right|\right)dx\\&\leq \frac{1}{|\Omega_\lambda|}\int_{\Omega_\lambda}\int_{\Omega_\lambda}\Phi(|u(x)-u(y)|)dydx\\&\leq \frac{1}{|\Omega_\lambda|}\int_{\Omega_\lambda}\int_{\Omega_\lambda}\Phi\left(\lambda^s\text{d}^s\frac{|u(x)-u(y)|}{|x-y|^s}\right)\frac{(\lambda\text{d})^N}{|x-y|^N}dxdy\\&= \frac{(\lambda\text{d})^N}{\lambda^N|\Omega|}\int_{\Omega_\lambda}\int_{\Omega_\lambda}\Phi\left(\lambda^s\text{d}^s\frac{|u(x)-u(y)|}{|x-y|^s}\right)d\mu\\&\leq \frac{\text{d}^N}{|\Omega|}\max\left\{\text{d}^{sp^-_\Phi},\text{d}^{sp^+_\Phi}\right\}\int_{\Omega_\lambda}\int_{\Omega_\lambda}\Phi\left(\lambda^s|D_su(x,y)|\right)d\mu,
    \end{align*}
   where the last inequality follows from \eqref{H3}. This completes the proof.
\end{proof}
\smallskip

In the following  lemma, we prove two main inequalities that we required in the proof of the main theorems. For $R>0$ and $k\in \mathbb{Z},$ define $$A_k(R)=\{x\in \mathbb{R}^N:2^kR\le |x|< 2^{k+1}R\}.$$
 
\begin{lemma}\label{Lemma-1-2} Let $N\geq 1,s\in (0,1),\,R>0,\,\alpha_1,\alpha_2\in \mathbb{R},$ and $\gamma=s-\alpha_1-\alpha_2$.  Then for any Orlicz function $\Phi$, there exists a positive constant  $C=C(s,\alpha_1,\alpha_2,N,R,\Phi)$ so that for every $ u\in \mathcal{C}^1_c(\mathbb{R}^N)$ and $k\in \mathbb{Z}$, the following inequalities hold:

\begin{equation*}
     (i)\;\quad \int\limits_{A_k(R)}  \Phi\left(\frac{|u(x)|}{|x|^\gamma}\right)dx \leq C2^{kN}\Phi\left(2^{-k\gamma}\big|(u)_{A_k(R)}\big|\right)\\+ C \int\limits_{A_k(R)}\int\limits_{A_k(R)}\Phi\left(|x|^{\alpha_1}|y|^{\alpha_2}|D_su|\right)d\mu,
\end{equation*}
\begin{equation*}
   (ii)\; \Phi\left(2^{-k\gamma}\left|(u)_{A_k(R)}- (u)_{A_{k+1}(R)}\right|\right)\leq \frac{C}{2^{kN}}\int\limits_{A_k(R)\cup A_{k+1}(R)}\int\limits_{A_k(R)\cup A_{k+1}(R)}\Phi\left(|x|^{\alpha_1}|y|^{\alpha_2}|D_su|\right)d\mu.
\end{equation*}
\end{lemma}
 \begin{proof}   
Let $ u\in \mathcal{C}^1_c(\mathbb{R}^N)$ and $k\in\mathbb{Z}$. We denote $A_k(R)$ by $A_k$ for simplicity.

\noindent$(i)$
Applying Proposition \ref{lem-orlicz} with $\Omega=\{x\in \mathbb{R}^N:R<|x|<2R\}$, $\lambda=2^k$ and observing that $(u)_{\Omega_\lambda}=(u)_{A_k}$, we obtain
\begin{equation*}
    \int_{A_k}\Phi(|u(x)-(u)_{A_k}|)dx\leq  C \int_{A_k}\int_{A_k}\Phi\left(2^{ks}|D_su(x,y)|\right)d\mu,
\end{equation*}
where $C=C(s,N,R,\Phi)$ is a positive constant. Thus, by \eqref{H2} we obtain 
\begin{align*}
  \int_{A_k}\Phi(|u(x)|)dx= &\int_{A_k}\Phi(|(u)_{A_k}+u(x)-(u)_{A_k}|)dx\\&\leq  2^{p^+_\Phi}\int_{A_k}\Phi(|(u)_{A_k}|)dx+2^{p^+_\Phi}\int_{A_k}\Phi(|u(x)-(u)_{A_k}|)dx\\& \leq 2^{p^+_\Phi}|A_k|\Phi\left(|(u)_{A_k}|\right)+ 2^{p^+_\Phi} C\int_{A_k}\int_{A_k}\Phi\left(2^{ks}|D_su(x,y)|\right)d\mu.
\end{align*}
Now replace $u$ by $2^{-k\gamma}u$ in the above inequality to obtain
\begin{multline}\label{eqn-lem2}
     \int_{A_k}\Phi\left(2^{-k\gamma}|u(x)|\right)dx\leq 2^{p^+_\Phi}|A_k|\Phi\left(2^{-k\gamma }|(u)_{A_k}|\right)\\+ 2^{p^+_\Phi}C\int_{A_k}\int_{A_k}\Phi\left(2^{k(s-\gamma)}|D_su(x,y)|\right)d\mu.
\end{multline}
If  $\gamma>0$, then by \eqref{H3} for every $ x \in A_k,$ we get $$\Phi\left(\frac{|u(x)|}{|x|^{\gamma}}\right)\leq \Phi\left(\frac{|u(x)|}{2^{k\gamma}R^{\gamma}}\right)
\leq  \max\left\{R^{-\gamma p^-_\Phi},R^{-\gamma p^+_\Phi}\right\}\Phi(2^{-k\gamma}|u(x)|).$$ 
On the other hand, if $\gamma\leq  0,$ then again by  \eqref{H3} for every $x \in A_k$, we get 
$$\Phi\left(\frac{|u(x)|}{|x|^{\gamma}}\right)< \Phi\left(\frac{|u(x)|}{2^{(k+1)\gamma}R^{\gamma}}\right)\leq \max\left\{(2R)^{-\gamma p^-_\Phi},(2R)^{-\gamma p^+_\Phi}\right\}\Phi(2^{-k\gamma}|u(x)|).$$ 
 Consequently, for all $\gamma\in \mathbb{R}$ we get a constant $C_1=C_1(\gamma,R,\Phi)>0$ such that
\begin{equation}\label{eqn-dom1}
    \Phi\left(\frac{|u(x)|}{|x|^{\gamma}}\right)\leq C_1\Phi(2^{-k\gamma}|u(x)|),\hspace{3mm} \forall\, x \in A_k.
\end{equation}
Similarly, by  \eqref{H3}, there exists  a constant $C_2=C_2(\alpha_1,\alpha_2,R,\Phi)>0$ such that for all $x, ~y\in A_k,$
\begin{equation}\label{eqn-dom}
    \Phi\left(2^{k(s-\gamma)}|D_su(x,y)|\right)= \Phi\left(2^{k(\alpha_1+\alpha_2)}|D_su(x,y)|\right)\leq C_2\Phi\left(|x|^{\alpha_1}|y|^{\alpha_2}|D_su(x,y)|\right).
\end{equation}
Hence, the result follows from \eqref{eqn-lem2}, \eqref{eqn-dom1}, and \eqref{eqn-dom}. 
\smallskip

\noindent $(ii)$
Let $z=(u)_{A_k\cup A_{k+1}}$. Then, by  \eqref{H2} and the Jenson's inequality, 
we get
\begin{align*}
     \Phi\left(|(u)_{A_k}-(u)_{A_{k+1}} |\right)=&\Phi\left(|(u)_{A_k}-z-(u)_{A_{k+1}}+z |\right)\\\le& 2^{p^+_\Phi}\Phi\left(|(u)_{A_k}-z|\right)+2^{p^+_\Phi}\Phi\left(|(u)_{A_{k+1}}-z|\right)\\\le&2^{p^+_\Phi}\Phi\left( \frac{1}{|A_k|}\int_{A_{k}}|u(x)-z|dx\right)+2^{p^+_\Phi}\Phi\left(\frac{1}{|A_{k+1}|}\int_{ A_{k+1}}|u(x)-z|dx\right)\nonumber\\\leq & \frac{2^{p^+_\Phi}}{|A_k|}\int_{A_k}\Phi(|u(x)-z|)dx+\frac{2^{p^+_\Phi}}{|A_{k+1}|}\int_{ A_{k+1}}\Phi(|u(x)-z|)dx\nonumber\\\leq &  \frac{2^{p^+_\Phi+1}}{|A_k|}\int_{A_k\cup A_{k+1}}\Phi(|u(x)-z|)dx.
\end{align*}
 Applying Proposition \ref{lem-orlicz} with $\Omega=\{x\in \mathbb{R}^N:R<|x|<4R\}$ and $\lambda=2^k$, we get
\begin{equation*}
    \int_{A_k\cup A_{k+1}}\Phi(|u(x)-(u)_{A_k\cup A_{k+1}}|)dx\leq C \int_{A_k\cup A_{k+1}}\int_{A_k\cup A_{k+1}}\Phi\left(2^{ks}|D_su(x,y)|\right)d\mu,
\end{equation*}
where $C=C(s,N,R,\Phi)$ is a positive constant. 
By combining the above two inequalities, we obtain 
\begin{equation*}
     \Phi\left(|(u)_{A_k}-(u)_{A_{k+1}} |\right)\leq \frac{2^{p^+_\Phi+1} C}{|A_k|}\int_{A_k\cup A_{k+1}}\int_{A_k\cup A_{k+1}}\Phi\left(2^{ks}|D_su(x,y)|\right)d\mu.
\end{equation*}
Now replace $u$ by $2^{-k\gamma}u$ to get
\begin{equation*}
    \Phi\left(2^{-k\gamma}|(u)_{A_k}-(u)_{A_{k+1}} |\right)\leq  \frac{2^{p^+_\Phi+1} C}{|A_k|}\int_{A_k\cup A_{k+1}}\int_{A_k\cup A_{k+1}}\Phi\left(2^{k(s-\gamma)}|D_su(x,y)|\right)d\mu.
\end{equation*}
Further, using  $\gamma=s-\alpha_1-\alpha_2$ and  \eqref{H3}, we get a constant $C_1=C_1(\alpha_1,\alpha_2,R,\Phi)>0$ (see \eqref{eqn-dom}) such that 
$$ \Phi\left(2^{k(s-\gamma)}|D_su(x,y)|\right)\leq C_1\Phi\left(|x|^{\alpha_1}|y|^{\alpha_2}|D_su(x,y)|\right),\;\;\;\forall\,x,y\in A_k\cup A_{k+1}.$$
 Hence, the conclusion follows from the above two inequalities.
\end{proof}

\section{Weighted fractional Orlicz-Hardy inequalities}
In this section, we prove Theorem \ref{Orlicz-Hardy}, Theorem \ref{Orlicz-Hardy2}, and Theorem \ref{log-orlicz}. First, we establish a lemma that plays an important role in proving the aforementioned theorems. 

\begin{lemma}\label{Lem-H}
    Let $\Phi$ be an Orlicz function and $\Lambda>1$.
Then there exists  $C=C(\Phi,\Lambda)>0$ such that 
     \begin{equation*}
 \Phi(a+b)\leq \lambda\Phi(a)+\frac{C}{(\lambda-1)^{p^+_\Phi-1}}\Phi(b),\;\;\;\forall\,a,b\in [0,\infty),\,\forall\,\lambda\in (1,\Lambda]. 
  \end{equation*}
  \end{lemma}
    \begin{proof} If $a=0$ or $b=0$, then the above inequality follows with $C=(\Lambda-1)^{p^+_\Phi-1}$. For $a\in (0,\infty)$ and $\lambda\in (1,\Lambda],$ consider the function $$f_{a,\lambda}(b)=\frac{\Phi(a+b)-\lambda\Phi(a)}{\Phi(b)},\;\;\;\,b\in (0,\infty).$$
 It is enough to show that there exists $C=C(\Phi,\Lambda)>0$ so that  
 \begin{equation}\label{eqn1-lem2.3}
     f_{a,\lambda}(b)\leq \frac{C}{(\lambda-1)^{p^+_\Phi-1}},\;\;\;\forall\,b\in (0,\infty).
 \end{equation}
 First, we provide estimates of $f_{a,\lambda}$ for certain values of $b.$
    For $b\geq a,$  we use \eqref{H3} to estimate
\begin{equation}\label{eqn-lem2.3}
    f_{a,\lambda}(b)\leq  \frac{\Phi(2b)-\lambda\Phi(a)}{\Phi(b)} \leq  \frac{\Phi(2b)}{\Phi(b)}\leq  \frac{2^{p^+_\Phi}\Phi(b)}{\Phi(b)}\leq\frac{2^{p^+_\Phi}(\Lambda-1)^{p^+_\Phi-1}}{(\lambda-1)^{p^+_\Phi-1}},\;\;\;\forall\,\lambda\in (1,\Lambda].
\end{equation}
Observe that, for fixed $a$ and $\lambda,$ as a function of $b$, $\Phi(a+b)-\lambda\Phi(a)$ is continuous, strictly increasing, and takes negative values near zero and positive values near infinity. Thus,  there exists a unique $\xi=\xi(a,\lambda)>0$ such that
\begin{equation}\label{xi}
    \Phi(a+\xi)-\lambda\Phi(a)=0.
\end{equation}
Therefore,
\begin{equation}\label{eq:negative}
    f_{a,\lambda}(b)\leq 0,\quad \forall \,b\in (0,\xi].
\end{equation}
Now we consider the two cases:  $(i)\, \xi\geq a, \quad (ii)\, \xi< a .$ 

\noindent $\underline{\xi\geq a}:$ In this case,   \eqref{eqn1-lem2.3}  follows easily from \eqref{eqn-lem2.3} and \eqref{eq:negative} with $C=2^{p^+_\Phi}(\Lambda-1)^{p^+_\Phi-1}$.
\smallskip

\noindent $\underline{\xi< a}:$  In this case, by \eqref{eqn-lem2.3} and \eqref{eq:negative}, it remains to estimate $f_{a,\lambda}(b)$  for $b\in(\xi,a)$. First, we  observe that
      $$\Phi(a+b)-\lambda\Phi(a)\leq \Phi(a+b)-\Phi(a)=\int_a^{a+b}\varphi(t)dt\leq b\varphi(a+b),$$ 
where the last inequality follows as  $\varphi$ is non-decreasing (see \eqref{right-deri}). Thus, using  \eqref{H1} and \eqref{H6} we get
\begin{equation}\label{eqn-2.10}
     f_{a,\lambda}(b)\leq \frac{b\varphi(a+b)}{\Phi(b)}\leq C_1\frac{b\varphi(2a)}{b\varphi(b)}\leq C_1C_2\frac{\varphi(a)}{\varphi(b)}\leq C_1C_2\frac{\varphi(a)}{\varphi(\xi)},\;\;\;\forall\,b\in(\xi,a),
\end{equation}
where $C_1,C_2$ are positive constants that depend only on $\Phi.$
From \eqref{xi}, we get $\xi=\Phi^{-1}(\lambda\Phi(a))-a$ and hence by using \eqref{H4}, we obtain
\begin{equation*}
   \xi=\Phi^{-1}(\lambda\Phi(a))-a\geq \min\left\{\lambda^{1/p^-_\Phi},\lambda^{1/p^+_\Phi}\right\}a-a=\lambda^{1/p^+_\Phi}a-a, \quad \forall\, \lambda>1. 
\end{equation*} 
Thus, by \eqref{H6} we get a constant $C_3=C_3(\Phi)>0$ such that
\begin{equation*}
 \varphi(\xi)\ge \varphi\left(\big(\lambda^{1/p^+_\Phi}-1\big)a\right)\geq C_3\min\left\{\big(\lambda^{1/p^+_\Phi}-1\big)^{p^-_\Phi-1},\big(\lambda^{1/p^+_\Phi}-1\big)^{p^+_\Phi-1}\right\}\varphi(a).
     \end{equation*}
Consequently, from \eqref{eqn-2.10} we obtain $$ f_{a,\lambda}(b)\leq  \frac{C_1C_2}{C_3h(\lambda)}= \frac{C_1C_2}{C_3(\lambda-1)^{p^+_\Phi-1}} \times\frac{(\lambda-1)^{p^+_\Phi-1}}{h(\lambda)},\;\;\;\forall\,b\in (\xi,a),$$ where
$$h(\lambda):=\min\left\{(\lambda^{1/p^+_\Phi}-1)^{p^-_\Phi-1},(\lambda^{1/p^+_\Phi}-1)^{p^+_\Phi-1}\right\}.$$
Next we show that $g(\lambda):=\frac{(\lambda-1)^{p^+_\Phi-1}}{h(\lambda)}$ is bounded on $(1,\Lambda].$
 Notice that,
 \[g(\lambda)=
     \begin{cases}
     \text{$\frac{(\lambda-1)^{p^+_\Phi-1}}{(\lambda^{1/p^+_\Phi}-1)^{p^+_\Phi-1}}$} & \quad\text{if $1<\lambda<2^{p^+_\Phi}$},\\ \text{$\frac{(\lambda-1)^{p^+_\Phi-1}}{(\lambda^{1/p^+_\Phi}-1)^{p^-_\Phi-1}}$} &\quad\text{if $\lambda\ge 2^{p^+_\Phi}$.}
     \end{cases}
\]
It is not hard to verify that $\lim_{\lambda\to 1+}g(\lambda)<\infty.$ Thus,  there exists $C_4=C_4(\Phi,\Lambda)>0$ such that  $g(\lambda)\leq C_4$ for all $\lambda\in (1,\Lambda].$ Therefore,
$$ f_{a,\lambda}(b) \leq \frac{C_1C_2C_4}{C_3(\lambda-1)^{p^+_\Phi-1}},\;\;\;\forall\,b\in (\xi,a).$$ 
This completes the proof.
     \end{proof} 
\smallskip

The next lemma proves the weighted fractional Orlicz-Hardy inequality \eqref{Hardy} for some range of $\gamma.$ This lemma is useful in proving Theorem \ref{Orlicz-Hardy}.
\begin{lemma} Let $N\geq 1,\,s\in (0,1)$, $\alpha_1, \, \alpha_2\in \mathbb{R}$,  and let 
$\gamma:=s-\alpha_1-\alpha_2$. For an Orlicz function $\Phi,$ if $\gamma<N/p^+_\Phi,$ then \eqref{Hardy} holds. Furthermore, if $\gamma\geq N/p^-_\Phi$, then \eqref{Hardy} fails.
\end{lemma}
\begin{proof}
Let $u\in \mathcal{C}_c^1(\mathbb{R}^N)$ and  $k\in \mathbb{Z}$. Choose $n_0\in \mathbb{Z}$ such that $supp(u)\subset B_{2^{n_0+1}}(0)$. Recall that $A_k(R)=\{x\in \mathbb{R}^N:2^kR\le |x|< 2^{k+1}R\}$. For simplicity, denote $A_k(1)$ by $A_k$. Now, apply Lemma \ref{Lemma-1-2} with $R=1$ to  get
\smallskip
\begin{align*}
\int_{A_k}\Phi\left(\frac{|u(x)|}{|x|^{\gamma}}\right)dx  \leq &C\left\{2^{kN}\Phi\left(2^{-k\gamma}|(u)_{A_k}|\right)+ \int_{A_k}\int_{A_k}\Phi\left(|x|^{\alpha_1}|y|^{\alpha_2}|D_su(x,y)|\right)d\mu\right\},
\end{align*} 
where $C=C(s,\alpha_1,\alpha_2,N,\Phi)$ is a positive constant.
Let $m\in \mathbb{Z}$ such that $m < n_0$. Summing the above
inequalities from $m$ to $n_0$, we get
\begin{multline}\label{eqn-thm1*}
     \int_{\{|x|\geq 2^m\}}  \Phi\left(\frac{|u(x)|}{|x|^\gamma}\right)dx=\sum_{k=m}^{n_0}\int_{A_k}  \Phi\left(\frac{|u(x)|}{|x|^\gamma}\right)dx\leq C\sum_{k=m}^{n_0}2^{kN}\Phi\left(2^{-k\gamma}|(u)_{A_k}|\right)\\+ C \int_{\mathbb{R}^N}\int_{\mathbb{R}^N}\Phi\left(|x|^{\alpha_1}|y|^{\alpha_2}|D_su(x,y)|\right)d\mu.
\end{multline}
Next, we estimate the first term of the right-hand side of the above inequality. Let $\Lambda>1$ be a number whose value will be chosen later. Then, by triangular inequality and Lemma \ref{Lem-H} with $\lambda=\Lambda$, there exists a positive constant $C_1=C_1(\Phi,\Lambda)$ such that
\begin{align*}
    \Phi\left(2^{-k\gamma}|(u)_{A_k}|\right)&\leq\Phi \left(2^{-k\gamma}|(u)_{A_{k+1}}|+2^{-k\gamma}\left|(u)_{A_k}- (u)_{A_{k+1}}\right|\right)\\&\leq \Lambda\Phi\left(2^{-k\gamma}|(u)_{A_{k+1}}|\right)+C_1\Phi \left(2^{-k\gamma}\left|(u)_{A_k}- (u)_{A_{k+1}}\right|\right),\;\;\;\forall\,k\in \mathbb{Z}.
\end{align*}
It follows from \eqref{H3} that 
$$\Phi\left(2^{-k\gamma}|(u)_{A_{k+1}}|\right)=\Phi\left(2^{\gamma }\cdot 2^{-(k+1)\gamma}|(u)_{A_{k+1}}|\right)\leq  \max\left\{2^{\gamma p^-_\Phi},2^{\gamma p^+_\Phi}\right\}\Phi\left(2^{-(k+1)\gamma}|(u)_{A_{k+1}}|\right).$$
Further, applying Lemma \ref{Lemma-1-2} with $R=1$ we get a positive constant $C_2=C_2(s,\alpha_1,\alpha_2,N,\Phi)$ such that
\begin{align*}
     \Phi\left(2^{-k\gamma}\left|(u)_{A_k}- (u)_{A_{k+1}}\right|\right)\leq \frac{C_2}{2^{kN}}\int_{A_k\cup A_{k+1}}\int_{A_k\cup A_{k+1}}\Phi\left(|x|^{\alpha_1}|y|^{\alpha_2}|D_su(x,y)|\right)d\mu.
  \end{align*}  
By combining the above there inequalities, we obtain
\begin{multline*}
\Phi\left(2^{-k\gamma}|(u)_{A_k}|\right)\leq \Lambda\max\left\{2^{\gamma p^-_\Phi },2^{\gamma p^+_\Phi}\right\}\Phi\left(2^{-(k+1)\gamma}|(u)_{A_{k+1}}|\right) \\+ \frac{C_1C_2}{ 2^{kN}}\int_{A_k\cup A_{k+1}}\int_{A_k\cup A_{k+1}}\Phi\left(|x|^{\alpha_1}|y|^{\alpha_2}|D_su(x,y)|\right)d\mu.
\end{multline*}
Multiply both sides of the above inequality by $2^{kN}$  to obtain
\begin{multline*}
2^{kN}\Phi\left(2^{-k\gamma}|(u)_{A_k}|\right)\leq \Lambda 2^{-N}\max\left\{2^{\gamma p^-_\Phi},2^{\gamma p^+_\Phi}\right\} 2^{(k+1)N}\Phi\left(2^{-(k+1)\gamma}|(u)_{A_{k+1}}|\right) \\+ C_1 C_2\int_{A_k\cup A_{k+1}}\int_{A_k\cup A_{k+1}}\Phi\left(|x|^{\alpha_1}|y|^{\alpha_2}|D_su(x,y)|\right)d\mu.
\end{multline*}
Therefore,
 \begin{multline}\label{eqn-thm1}
\sum_{k=m}^{n_0}2^{kN}\Phi\left(2^{-k\gamma}|(u)_{A_k}|\right)\leq \Lambda 2^{-N}\max\left\{2^{\gamma p^-_\Phi },2^{\gamma p^+_\Phi}\right\}\sum_{k=m}^{n_0} 2^{(k+1)N}\Phi\left(2^{-(k+1)\gamma}|(u)_{A_{k+1}}|\right) \\+ C_1 C_2\int_{\mathbb{R}^N}\int_{\mathbb{R}^N}\Phi\left(|x|^{\alpha_1}|y|^{\alpha_2}|D_su(x,y)|\right)d\mu.
\end{multline}
Since $u\equiv 0$ on $A_{n_0+1}$, we have $(u)_{A_{n_0+1}}=0.$ Thus by re-indexing, we get
\begin{equation*}
    \sum_{k=m}^{n_0} 2^{(k+1)N}\Phi\left(2^{ -(k+1)\gamma}|(u)_{A_{k+1}}|\right)=\sum_{k=m+1}^{n_0} 2^{kN}\Phi\left(2^{ -k\gamma}|(u)_{A_{k}}|\right)\le \sum_{k=m}^{n_0} 2^{kN}\Phi\left(2^{ -k\gamma}|(u)_{A_{k}}|\right).
\end{equation*}
Consequently, from \eqref{eqn-thm1} we obtain
 \begin{multline}\label{eqn-thm1-1}
\sum_{k=m}^{n_0}2^{kN}\Phi\left(2^{-k\gamma}|(u)_{A_k}|\right)\leq \Lambda 2^{-N}\max\left\{2^{\gamma p^-_\Phi},2^{\gamma p^+_\Phi}\right\}\sum_{k=m}^{n_0} 2^{kN}\Phi\left(2^{-k\gamma}|(u)_{A_k}|\right) \\+ C_1 C_2\int_{\mathbb{R}^N}\int_{\mathbb{R}^N}\Phi\left(|x|^{\alpha_1}|y|^{\alpha_2}|D_su(x,y)|\right)d\mu.
\end{multline}
If $\gamma>0$, then  $\max\left\{2^{\gamma p^-_\Phi},2^{\gamma p^+_\Phi}\right\}=2^{\gamma p^+_\Phi}$. 
 Since $\gamma<N/p^+_\Phi$, there exists  
 $\Lambda=\Lambda(N,\gamma,p^+_\Phi)>1$ such that $$\Lambda 2^{-N}2^{\gamma p^+_\Phi}<1.$$
On the other hand, if $\gamma\leq 0$ then $\max\left\{2^{\gamma p^-_\Phi },2^{\gamma p^+_\Phi}\right\}=2^{\gamma p^-_\Phi}$ and there exists $\Lambda=\Lambda(N,\gamma,p^-_\Phi)>1$ such that
 $\Lambda2^{-N}2^{\gamma p^-_\Phi}<1.$
 Therefore, for all $\gamma\in \mathbb{R}$ we get a constant $\Lambda=\Lambda(N,\gamma,\Phi)>1$ such that $$C_3:=\Lambda 2^{-N}\max\left\{2^{\gamma p^-_\Phi },2^{\gamma p^+_\Phi}\right\}<1.$$
Now  from \eqref{eqn-thm1-1} we obtain
 \begin{equation*}
\left(1-C_3\right)\sum_{k=m}^{n_0}2^{kN}\Phi\left(2^{-k\gamma}|(u)_{A_k}|\right)\leq   C_1C_2 \int_{\mathbb{R}^N}\int_{\mathbb{R}^N}\Phi\left(|x|^{\alpha_1}|y|^{\alpha_2}|D_su(x,y)|\right)d\mu.
\end{equation*}
Thus, from \eqref{eqn-thm1*}  we get
 \begin{align*}
   \int_{\{|x|\ge 2^m\}}  \Phi\left(\frac{|u(x)|}{|x|^\gamma}\right)dx\leq C\left(1+\frac{C_1C_2}{1-C_3}\right) \int_{\mathbb{R}^N}\int_{\mathbb{R}^N}\Phi\left(|x|^{\alpha_1}|y|^{\alpha_2}|D_su(x,y)|\right)d\mu.
\end{align*}
Hence, \eqref{Hardy} follows by taking 
$m\to-\infty$.
\smallskip

\noindent Now, we prove the second part of the theorem.
Recall that, $B_1(0)$ is the open ball centred at the origin with radius $1$. By using \eqref{H3} and $\gamma>0$, we get
\begin{align}\label{eqn-fails}
\int_{B_1(0)}\Phi\left(\frac{|u(x)|}{|x|^\gamma}\right)dx &\geq  \int_{B_1(0)}\min\left\{\frac{1}{|x|^{\gamma p^-_\Phi}},\frac{1}{|x|^{\gamma p^+_\Phi}}\right\}\Phi\left(|u(x)|\right)dx\nonumber \\& =\int_{B_1(0)}\frac{\Phi\left(|u(x)|\right)}{|x|^{\gamma p^-_\Phi}}dx,\;\;\;\forall\,u\in \mathcal{C}_c^1(\mathbb{R}^N).
\end{align}
Now if  $\gamma\geq N/p^-_\Phi $, then  $1/|x|^{\gamma p^-_\Phi}$ is not locally integrable on $\mathbb{R}^N$. Thus, from \eqref{eqn-fails}, we can conclude that \eqref{Hardy} fails for any $u\in \mathcal{C}_c^1(\mathbb{R}^N)$ with  $u\equiv 1$ on $B_1(0)$. This completes the proof. 
\end{proof}

\noindent\textbf{Proof of Theorem \ref{Orlicz-Hardy}.} For $\gamma < N/p_\Phi^\oplus,$ we have  $p_\Phi^\oplus<N/\gamma.$ Thus, as   $p_\Phi^\oplus$ being an infimum, there exists $\Psi\asymp\Phi$ such that $$p_\Phi^\oplus\le p^+_\Psi<N/\gamma.$$
Now apply the above lemma for $\Psi$ and use  $\Phi \asymp\Psi$ to conclude that \eqref{Hardy} holds for $\Phi$ for $\gamma < N/p_\Phi^\oplus$.

\noindent On the other hand, for  $\gamma>N/p^\ominus_\Phi,$
we use the definition of $p^\ominus_\Phi$ to get  $\Psi\asymp\Phi$ such that $$p_\Phi^\ominus\ge p^-_\Psi>N/\gamma.$$
 Therefore, applying Theorem \ref{Orlicz-Hardy} for $\Psi$ and using $\Phi \asymp\Psi$ we conclude that \eqref{Hardy} fails  for  $\gamma>N/p^\ominus_\Phi$. If $p_\Phi^\ominus$ is attained at some $\Psi \asymp \Phi,$ then by applying the above lemma for $\Psi$ and using $\Psi \asymp \Phi,$ we conclude that \eqref{Hardy} fails  also at $\gamma=N/p^\ominus_\Phi.$
 \qed
 
 \smallskip
  
Next, we prove an important lemma. 
\begin{lemma}
 Let $N\geq 1,\, s\in (0,1)$, ~$\alpha_1, ~\alpha_2\in \mathbb{R}$,  and let  $\gamma:=s-\alpha_1-\alpha_2$. For an  Orlicz function $\Phi$, if $\gamma>N/p^-_\Phi$, then   \eqref{Hardy-0} holds.
\end{lemma}
\begin{proof}
Let $u\in \mathcal{C}_c^1(\mathbb{R}^N\setminus\{0\})$ and  $k\in \mathbb{Z}$. Choose $n_0,m_0\in \mathbb{Z}$ with $m_0<n_0$ such that $supp(u)\cap B_{2^{m_0}}(0)=\emptyset$ and $supp(u)\subset B_{2^{n_0+1}}(0).$ Recall that $A_k(R)=\{x\in \mathbb{R}^N:2^kR\le |x|<2^{k+1}R\},\;R>0.$ For simplicity, let's denote $A_k(1)$ as $A_k$. Then, we have $A_k(1/2)=A_{k-1}.$  Applying Lemma \ref{Lemma-1-2} with $R=1$,
we get
  \begin{align*}
   \int_{\mathbb{R}^N}  \Phi\left(\frac{|u(x)|}{|x|^\gamma}\right)dx&=\sum_{k=m_0}^{n_0}\int_{A_k}  \Phi\left(\frac{|u(x)|}{|x|^\gamma}\right)dx\\& \leq C\sum_{k=m_0}^{n_0}\left\{2^{kN}\Phi\left(2^{-k\gamma}|(u)_{A_k}|\right)+ \int_{A_k}\int_{A_k}\Phi\left(|x|^{\alpha_1}|y|^{\alpha_2}|D_su(x,y)|\right)d\mu\right\}\\& \leq C\sum_{k=m_0}^{n_0}2^{kN}\Phi\left(2^{-k\gamma}|(u)_{A_k}|\right)+ C \int_{\mathbb{R}^N}\int_{\mathbb{R}^N}\Phi\left(|x|^{\alpha_1}|y|^{\alpha_2}|D_su(x,y)|\right)d\mu,
\end{align*}
where $C=C(s,\alpha_1,\alpha_2,N,\Phi)$ is a positive constant. Thus, it is sufficient to show that there exists a  positive constant $C_1$ independent of $u$  such that
  \begin{align}\label{thm3-eq1}
\sum_{k=m_0}^{n_0}2^{kN}\Phi\left(2^{-k\gamma}|(u)_{A_k}|\right)\leq  C_1 \int_{\mathbb{R}^N}\int_{\mathbb{R}^N}\Phi\left(|x|^{\alpha_1}|y|^{\alpha_2}|D_su(x,y)|\right)d\mu.
\end{align}
Let $\Lambda>1$ be fixed, which will be chosen later. By triangular inequality and Lemma \ref{Lem-H} with $\lambda=\Lambda$, we get a positive constant $C_1=C_1(\Phi,\Lambda)$ such that
\begin{align*}
    \Phi\left(2^{-k\gamma}|(u)_{A_{k}}|\right)&\leq \Phi\left(2^{-k\gamma}|(u)_{A_{k-1}}|+2^{-k\gamma}\left|(u)_{A_{k-1}}-(u)_{A_{k}}\right|\right)\\&\leq \Lambda\Phi\left(2^{-k\gamma}|(u)_{A_{k-1}}|\right)+C_1\Phi \left(2^{-k\gamma}\left|(u)_{A_{k-1}}- (u)_{A_{k}}\right|\right),\;\;\;\forall\, k\in \mathbb{Z}.
\end{align*}
 Since $\gamma>0$, it follows from \eqref{H3} that 
$$\Phi\left(2^{ -k\gamma}|(u)_{A_{k-1}}|\right)=\Phi\left(2^{-\gamma}\cdot2^{ -(k-1)\gamma}|(u)_{A_{k-1}}|\right)\leq  2^{-\gamma p^-_\Phi}\Phi\left(2^{-(k-1)\gamma}|(u)_{A_{k-1}}|\right).$$ 
Further, applying Lemma \ref{Lemma-1-2} with $R=1/2$, we get  a constant $C_2>0$ such that  
\begin{equation*}
\Phi\left(2^{-k\gamma}\left|(u)_{A_{k-1}}- (u)_{A_{k}}\right|\right)\leq \frac{C_2}{2^{kN}}\int_{A_{k-1}\cup A_{k}}\int_{A_{k-1}\cup A_{k}}\Phi\left(|x|^{\alpha_1}|y|^{\alpha_2}|D_su(x,y)|\right)d\mu.
  \end{equation*}  
By combining the above three inequalities yield
\begin{multline}\label{eqn-main2}
\Phi\left(2^{ -k\gamma}|(u)_{A_{k}}|\right)\leq \Lambda 2^{-\gamma p^-_\Phi}\Phi\left(2^{ -(k-1)\gamma}|(u)_{A_{k-1}}|\right) \\+ \frac{C_1C_2}{ 2^{kN}}\int_{A_{k-1}\cup A_{k}}\int_{A_{k-1}\cup A_{k}}\Phi\left(|x|^{\alpha_1}|y|^{\alpha_2}|D_su(x,y)|\right)d\mu.
\end{multline}
Multiply both sides of the above inequality by $2^{kN}$  to obtain
\begin{multline*}
2^{kN}\Phi\left(2^{ -k\gamma}|(u)_{A_{k}}|\right)\leq  \Lambda2^{N-\gamma p^-_\Phi}2^{(k-1)N}\Phi\left(2^{-(k-1)\gamma}|(u)_{A_{k-1}}|\right) \\+ C_1C_2 \int_{A_{k-1}\cup A_{k}}\int_{A_{k-1}\cup A_{k}}\Phi\left(|x|^{\alpha_1}|y|^{\alpha_2}|D_su(x,y)|\right)d\mu.
\end{multline*}
Summing the above inequalities  from  $m_0$ to $n_0$, we obtain 
 \begin{multline}\label{eqn-thm3-2}
\sum_{k=m_0}^{n_0}2^{kN}\Phi\left(2^{-k\gamma}|(u)_{A_{k}}|\right)\leq  \Lambda2^{N-\gamma p^-_\Phi} \sum_{k=m_0}^{n_0} 2^{(k-1)N}\Phi\left(2^{ -(k-1)\gamma}|(u)_{A_{k-1}}|\right) \\+ C_1C_2 \int_{\mathbb{R}^N}\int_{\mathbb{R}^N}\Phi\left(|x|^{\alpha_1}|y|^{\alpha_2}|D_su(x,y)|\right)d\mu.
\end{multline}
Since $u\equiv 0$ on $A_{m_0-1}$, we have $(u)_{A_{m_0-1}}=0$. Now, by re-indexing, we get
\begin{align*}
    \sum_{k=m_0}^{n_0} 2^{(k-1)N}\Phi\left(2^{ -(k-1)\gamma}|(u)_{A_{k-1}}|\right)=\sum_{k=m_0}^{n_0-1} 2^{kN}\Phi\left(2^{ -k\gamma}|(u)_{A_{k}}|\right)\leq \sum_{k=m_0}^{n_0} 2^{kN}\Phi\left(2^{ -k\gamma}|(u)_{A_{k}}|\right).
\end{align*}
Therefore, from \eqref{eqn-thm3-2} we obtain
\begin{equation*}
\left\{1- \Lambda2^{N-\gamma p^-_\Phi}\right\}\sum_{k=m_0}^{n_0}2^{kN}\Phi\left(2^{-k\gamma}|(u)_{A_k}|\right)
\leq    C_1C_2\int_{\mathbb{R}^N}\int_{\mathbb{R}^N}\Phi\left(|x|^{\alpha_1}|y|^{\alpha_2}|D_su(x,y)|\right)d\mu
\end{equation*}
Since $\gamma>N/p^-_\Phi$, there exists a constant $\Lambda=\Lambda(N,\gamma,p^-_\Phi)>1$ such that $$\Lambda2^{N-\gamma p^-_\Phi}<1.$$
Thus, \eqref{thm3-eq1} follows after observing  that $ 1>\Lambda2^{N-\gamma p^-_\Phi}$.
This completes the proof.
\end{proof}

\noindent\textbf{Proof of Theorem \ref{Orlicz-Hardy2}.} For $\gamma > N/p_\Phi^\ominus,$ we have  $p_\Phi^\ominus>N/\gamma.$ Thus, as   $p_\Phi^\ominus$ being a supremum, there exists $\Psi\asymp\Phi$ such that $$p_\Phi^\ominus\ge p^-_\Psi>N/\gamma.$$
Now apply the above lemma for $\Psi$ and use  $\Phi \asymp\Psi$ to conclude that \eqref{Hardy-0} holds for $\gamma >N/p_\Phi^\ominus$. 
\qed 
 \smallskip

\noindent\textbf{Proof of Theorem \ref{log-orlicz}:} Recall that $A_k(R)=\{x\in \mathbb{R}^N: 2^kR\leq |x|<2^{k+1}R\},\;k\in \mathbb{Z}$.  For simplicity, let's denote $A_k(R)$ as $A_k$. 

\noindent $(i)$ Let $p^\oplus_\Phi$ is attained at $\Psi\asymp\Phi$. Thus,  $p^\oplus_\Phi=p_\Psi^+$ and hence $\gamma=N/p^\oplus_\Phi$ implies $\gamma=N/p^+_\Psi$.
Let $u\in \mathcal{C}^1_c(\mathbb{R}^N)$ such that $supp(u)\subset B_R(0)$. We have 
$$B_R(0)=\bigcup_{k=-\infty}^{-1}A_k.$$ 
 Now for $x\in A_k$ with $k\in \mathbb{Z}^-=\{n\in \mathbb{Z}: n<0\},$ we have $$\log(2R/|x|)>-k\log 2\ge \log 2.$$
Therefore, 
$$\int_{A_k}\frac{\Psi\left(|x|^{-\gamma}|u(x)|\right)}{\big(\log(2R/|x|)\big)^{p^+_\Psi}}dx \leq \frac{1}{(-k)^{p^+_\Psi}(\log 2)^{p^+_\Psi}}\int_{A_k}\Psi\left(\frac{|u(x)|}{|x|^{\gamma}}\right)dx,\;\;\;\forall\,k\in\mathbb{Z}^-.$$
Further, by  Lemma \ref{Lemma-1-2}, we have
\begin{align*}
\int_{A_k}\Psi\left(\frac{|u(x)|}{|x|^{\gamma}}\right)dx  \leq C\left\{2^{kN}\Psi\left(2^{-k\gamma}|(u)_{A_k}|\right)+\int_{A_k}\int_{A_k}\Psi(|x|^{\alpha_1}|y|^{\alpha_2}|D_su(x,y)|)d\mu\right\}.
\end{align*} 
Consequently, using $\frac{1}{(-k)^{p^+_\Psi}}\le 1$ for every $k\in\mathbb{Z}^-,$ we get
\begin{multline*}
      \int_{A_k}\frac{\Psi\left(|x|^{-\gamma}|u(x)|\right)}{\big(\log(2R/|x|)\big)^{p^+_\Psi}}dx \leq \frac{C2^{kN}}{(-k)^{p^+_\Psi}}\Psi\left(2^{-k\gamma}|(u)_{A_k}|\right)+C \int_{A_k}\int_{A_k}\Psi(|x|^{\alpha_1}|y|^{\alpha_2}|D_su(x,y)|)d\mu,
\end{multline*}
where $C$ is a positive constant (independent of both $u$ and $k$). Summing the above inequalities from $m\in \mathbb{Z}^-$ to $-1,$ we obtain
\begin{multline}
     \int_{\{2^mR\le |x|< R\}}\frac{\Psi\left(|x|^{-\gamma}|u(x)|\right)}{\big(\log(2R/|x|)\big)^{p^+_\Psi}}dx \leq C\sum_{k=m}^{-1}\frac{2^{kN}}{(-k)^{p^+_\Psi}}\Psi\left(2^{-k\gamma}|(u)_{A_k}|\right)\\+C \int_{B_R(0)}\int_{B_R(0)}\Psi(|x|^{\alpha_1}|y|^{\alpha_2}|D_su(x,y)|)d\mu.\label{log-eq1*}
\end{multline} 
Next, we estimate the first term on the right-hand side of the above inequality.
 By triangular inequality and Lemma \ref{Lem-H} with $\Lambda=2^{p^+_\Psi}$, we get a constant $C_1=C_1(\Psi,\Lambda)>0$ satisfying \begin{align*}
\Psi\left(2^{-k\gamma}|(u)_{A_k}|\right)&\leq \Psi\left(2^{-k\gamma}|(u)_{A_{k+1}}|+2^{-k\gamma}\left|(u)_{A_k}-(u)_{A_{k+1}}\right|\right)\\&\leq \lambda\Psi\left(2^{-k\gamma}|(u)_{A_{k+1}}|\right)+\frac{C_1}{\left(\lambda-1\right)^{p^+_\Psi-1}}\Psi\left(2^{-k\gamma}\left|(u)_{A_k}- (u)_{A_{k+1}}\right|\right),
\end{align*}
for every $\lambda\in(1,2^{p^+_\Psi})$ and $k\in \mathbb{Z}.$ Now we use  \eqref{H3} and $\gamma=N/p^\oplus_\Phi=N/p^+_\Psi>0$ to yield
\begin{align*}
    \Psi\left(2^{-k\gamma}|(u)_{A_{k+1}}|\right)&=\Psi\left(2^\gamma\cdot 2^{-(k+1)\gamma}|(u)_{A_{k+1}}|\right)\nonumber\\&\leq  2^{\gamma p^+_\Psi}\Psi\left(2^{ -(k+1)\gamma}|(u)_{A_{k+1}}|\right)=  2^N\Psi\left(2^{ -(k+1)\gamma}|(u)_{A_{k+1}}|\right).
\end{align*} 
Moreover, by Lemma \ref{Lemma-1-2}, we have
\begin{align*}
     \Psi\left(2^{-k\gamma}\left|(u)_{A_k}- (u)_{A_{k+1}}\right|\right)\leq \frac{C_2}{2^{kN}}\int_{A_k\cup A_{k+1}}\int_{A_k\cup A_{k+1}}\Psi(|x|^{\alpha_1}|y|^{\alpha_2}|D_su(x,y)|)d\mu,
  \end{align*}  
where $C_2$ is a positive constant (independent of both $u$ and $k$). 
By combining the above three inequalities, for every $\lambda\in (1,2^{p^+_\Psi})$ and $k\in \mathbb{Z}$ we obtain
\begin{multline}\label{log-eqn4*}
\Psi\left(2^{-k\gamma}|(u)_{A_k}|\right)\leq 2^N\lambda\Psi\left(2^{-(k+1)\gamma}|(u)_{A_{k+1}}|\right) \\+ \frac{C_1C_2}{2^{kN}\left(\lambda-1\right)^{p^+_\Psi-1}}\int_{A_k\cup A_{k+1}}\int_{A_k\cup A_{k+1}}\Psi(|x|^{\alpha_1}|y|^{\alpha_2}|D_su(x,y)|)d\mu.
\end{multline}
 Now, for each $k\in \mathbb{Z}^-$, we choose 
 $$\lambda_k=\left(\frac{-k}{-k-1/2}\right)^{p^+_\Psi-1}.$$ For this choice of $\lambda_k$, one can verify that $$\lambda_k\in (1,2^{p^+_\Psi}) \text{ and }\frac{1}{(\lambda_k-1)}\asymp -k,\;\;\;\forall\,k\in \mathbb{Z}^-.$$
Thus, from \eqref{log-eqn4*}, for every $k\in \mathbb{Z}^-$ we obtain
\begin{multline*}
    \Psi\left(2^{-k\gamma}|(u)_{A_k}|\right)\leq 2^N\left(\frac{-k}{-k-1/2}\right)^{p^+_\Psi-1}\Psi\left(2^{-(k+1)\gamma}|(u)_{A_{k+1}}|\right)\\+C_3 \frac{(-k)^{p^+_\Psi-1}}{2^{kN}}  \int_{A_k\cup A_{k+1}}\int_{A_k\cup A_{k+1}}\Psi(|x|^{\alpha_1}|y|^{\alpha_2}|D_su(x,y)|)d\mu,
\end{multline*}
for some $C_3>0$.
This gives
\begin{multline*}
\frac{2^{kN}}{(-k)^{p^+_\Psi-1}}\Psi\left(2^{-k\gamma}|(u)_{A_k}|\right)\leq \frac{2^{(k+1)N}}{(-k-1/2)^{p^+_\Psi-1}}\Psi\left(2^{ -(k+1)\gamma}|(u)_{A_{k+1}}|\right)\\+C_3\int_{A_k\cup A_{k+1}}\int_{A_k\cup A_{k+1}}\Psi(|x|^{\alpha_1}|y|^{\alpha_2}|D_su(x,y)|)d\mu.
\end{multline*}
Summing the above inequalities from $m\in \mathbb{Z}^-$ to $-1,$  we obtain 
\begin{multline}\label{log-eqn2*}
\sum_{k=m}^{-1}\frac{2^{kN}}{(-k)^{p^+_\Psi-1}}\Psi\left(2^{-k\gamma}|(u)_{A_k}|\right)\leq \sum_{k=m}^{-1} \frac{2^{(k+1)N}}{(-k-1/2)^{p^+_\Psi-1}}\Psi\left(2^{ -(k+1)\gamma}|(u)_{A_{k+1}}|\right)\\+C_3\int_{\mathbb{R}^N}\int_{\mathbb{R}^N}\Psi\left(|x|^{\alpha_1}|y|^{\alpha_2}|D_su(x,y)|\right)d\mu.
\end{multline}
Since $u\equiv 0$ on $A_{0}$, we have $(u)_{A_0}=0$. Thus, by re-indexing, we get 
\begin{align*}
  \sum_{k=m}^{-1}\frac{2^{(k+1)N}}{(-k-1/2)^{p^+_\Psi-1}}\Psi\left(2^{ -(k+1)\gamma}|(u)_{A_{k+1}}|\right)&=\sum_{k=m+1}^{-1}\frac{2^{kN}}{(-k+1/2)^{p^+_\Psi-1}}\Psi\left(2^{-k\gamma}|(u)_{A_k}|\right)\\&\leq \sum_{k=m}^{-1}\frac{2^{kN}}{(-k+1/2)^{p^+_\Psi-1}}\Psi\left(2^{-k\gamma}|(u)_{A_k}|\right).  
\end{align*}
Therefore,  from \eqref{log-eqn2*} we get
\begin{multline*}
   \sum_{k=m}^{-1}\left\{\frac{2^{kN}}{(-k)^{p^+_\Psi-1}}-\frac{2^{kN}}{(-k+1/2)^{p^+_\Psi-1}}\right\}\Psi\left(2^{-k\gamma}|(u)_{A_k}|\right)\\\leq C_3\int_{\mathbb{R}^N}\int_{\mathbb{R}^N}\Psi\left(|x|^{\alpha_1}|y|^{\alpha_2}|D_su(x,y)|\right)d\mu.
\end{multline*}
Now, for  $k\in\mathbb{Z}^-$, we have
\begin{equation*}
    \frac{1}{(-k)^{p^+_\Psi-1}}-\frac{1}{(-k+1/2)^{p^+_\Psi-1}}\asymp  \frac{1}{(-k)^{p^+_\Psi}}.
\end{equation*}
Consequently, from \eqref{log-eq1*} we obtain
\begin{align*}
 \int_{\{2^mR\le|x|<R\}}\frac{\Psi\left(|x|^{-\gamma}|u(x)|\right)}{\big(\log(2R/|x|)\big)^{p^+_\Psi}}dx\leq C_4\int_{\mathbb{R}^N}\int_{\mathbb{R}^N}\Psi\left(|x|^{\alpha_1}|y|^{\alpha_2}|D_su(x,y)|\right)d\mu,
\end{align*}
where $C_4$ is a positive constant  independent of both $u$ and $m$. Hence, the required inequality 
follows by taking $m\to-\infty$ and using $\Phi\asymp\Psi$ and $p^+_\Phi=p_\Phi^\oplus$.
\smallskip

\noindent $(ii)$ Given that $p^\ominus_\Phi$ is attained at $\Psi\asymp\Phi$ and $\gamma=N/p^\ominus_\Phi$. Thus, $p^\ominus_\Phi=p_\Psi^-$ and hence $\gamma=N/p_\Psi^-.$
  Let $u\in \mathcal{C}^1_c(\mathbb{R}^N)$ such that $supp(u)\subset B_R(0)^c$. 
  Choose $n_0\in \mathbb{N}$  such that $supp(u)\subset B_{2^{n_0+1}R}(0).$ Notice that, $$B_R(0)^c=\bigcup_{k=0}^{\infty}A_k.$$ Now for $x\in A_k$ with $k\in \mathbb{N}\cup \{0\},$ we have $$\log(2|x|/R)\ge (k+1)\log 2\ge \log 2.$$  
Therefore, 
\begin{equation*}
\int_{A_k}\frac{\Psi\left(|x|^{-\gamma}|u(x)|\right)}{\big(\log(2|x|/R)\big)^{p^+_\Psi}}dx\leq \frac{1}{(k+1)^{p^+_\Psi}(\log 2)^{p^+_\Psi}}\int_{A_k}\Psi\left(\frac{|u(x)|}{|x|^{\gamma}}\right)dx.
\end{equation*}
Moreover, by Lemma \ref{Lemma-1-2}, we have
\begin{align*}
\int_{A_k}\Psi\left(\frac{|u(x)|}{|x|^{\gamma}}\right)dx \leq C\left\{2^{kN}\Psi\left(2^{-k\gamma}|(u)_{A_k}|\right)+ \int_{A_k}\int_{A_k}\Psi\left(|x|^{\alpha_1}|y|^{\alpha_2}|D_su(x,y)|\right)d\mu \right\}.
\end{align*} 
Thus, for $k\in \mathbb{N}\cup \{0\},$ we obtain
\begin{align*}
       \int_{A_k}\frac{\Psi\left(|x|^{-\gamma}|u(x)|\right)}{\big(\log(2|x|/R)\big)^{p^+_\Psi}}\le\frac{C2^{kN}}{(k+1)^{p^+_\Psi}}\Psi\left(2^{-k\gamma}|(u)_{A_k}|\right)+ C\int_{A_k}\int_{A_k}\Psi\left(|x|^{\alpha_1}|y|^{\alpha_2}|D_su(x,y)|\right)d\mu,
\end{align*} 
where $C$ is a positive constant  (independent of both $u$ and $k$). Summing the above inequalities from $0$ to $n_0$, we get
\begin{multline}
     \int_{B_R(0)^c}\frac{\Psi\left(|x|^{-\gamma}|u(x)|\right)}{\big(\log(2|x|/R)\big)^{p^+_\Psi}}dx=\sum_{k=0}^{n_0}\frac{\Psi\left(|x|^{-\gamma}|u(x)|\right)}{\big(\log(2|x|/R)\big)^{p^+_\Psi}}dx \leq C\sum_{k=0}^{n_0}\frac{2^{kN}}{(k+1)^{p^+_\Psi}}\Psi\left(2^{-k\gamma}|(u)_{A_k}|\right) \\+ C\int_{B_R(0)^c}\int_{B_R(0)^c}\Psi\left(|x|^{\alpha_1}|y|^{\alpha_2}|D_su(x,y)|\right)d\mu.\label{log-eq3}
\end{multline} 
Next, we estimate the first term on the right-hand side of \eqref{log-eq3}.  
By triangular inequality and Lemma \ref{Lem-H} with $\Lambda=2^{p^+_\Psi}$, there exists a positive constant $C_1=C_1(\Psi,\Lambda)$ such that for every $\lambda\in (1,2^{p^+_\Psi})$ and $k\in \mathbb{Z},$
\begin{align*}
\Psi\left(2^{-k\gamma}|(u)_{A_k}|\right)&\leq  \Psi\left(2^{-k\gamma}|(u)_{A_{k-1}}|+2^{-k\gamma}\left|(u)_{A_{k-1}}- (u)_{A_k}\right|\right)\\\leq &\lambda\Psi\left(2^{-k\gamma}|(u)_{A_{k-1}}|\right)+\frac{C_1}{\left(\lambda-1\right)^{p^+_\Psi-1}}\Psi\left(2^{-k\gamma}\left|(u)_{A_{k-1}}-(u)_{A_k}\right|\right).
\end{align*}
Since  $\gamma=N/p^-_\Psi>0$, using  \eqref{H3}, we observe that,
\begin{align*}
    \Psi\left(2^{-k\gamma}|(u)_{A_{k-1}}|\right)= \Psi\left(2^{-N/p^-_\Psi}\cdot 2^{ -(k-1)\gamma}|(u)_{A_{k-1}}|\right)\le  2^{-N}\Psi\left(2^{ -(k-1)\gamma}|(u)_{A_{k-1}}|\right).
\end{align*} 
Furthermore, applying Lemma \ref{Lemma-1-2} (replacing $R$ by $R/2$), we get a constant $C_2>0$  such that 
\begin{align*}
\Psi\left(2^{-k\gamma}\left|(u)_{A_{k-1}}- (u)_{A_k}\right|\right)&\leq \frac{C_2}{2^{kN}}\int_{A_{k-1}\cup A_k}\int_{A_{k-1}\cup A_k}\Psi\left(|x|^{\alpha_1}|y|^{\alpha_2}|D_su(x,y)|\right)d\mu.
  \end{align*}  
 By combining the above three inequalities, for every $\lambda\in (1,2^{p^+_\Psi})$ and $k\in \mathbb{Z}$ we obtain
\begin{multline}\label{log-eqn4**}
\Psi\left(2^{-k\gamma}|(u)_{A_k}|\right)\leq \lambda 2^{-N}\Psi\left(2^{-(k-1)\gamma}|(u)_{A_{k-1}}|\right) \\+ \frac{C_1C_2}{2^{kN}\left(\lambda-1\right)^{p^+_\Psi-1}}\int_{A_{k-1}\cup A_k}\int_{A_{k-1}\cup A_k}\Psi\left(|x|^{\alpha_1}|y|^{\alpha_2}|D_su(x,y)|\right)d\mu.
\end{multline}
Now, for each $k\in \mathbb{N}\cup \{0\}$, we choose
 $$\lambda_k=\left(\frac{k+1}{k+1/2}\right)^{p^+_\Psi-1}.$$ Then, one can verify that   $$\lambda_k\in (1,2^{p^+_\Psi})\;\:\text{and}\;\;\frac{1}{(\lambda_k-1)}\asymp (k+1),\;\;\;\forall\,k\in \mathbb{N}\cup \{0\}.$$
Consequently,  from \eqref{log-eqn4**}  we get
\begin{multline*}
\Psi\left(2^{ -k\gamma}|(u)_{A_k}|\right)\leq 2^{-N}\left(\frac{k+1}{k+1/2}\right)^{p^+_\Psi-1}\Psi\left(2^{-(k-1)\gamma}|(u)_{A_{k-1}}|\right)\\+C_3 \frac{(k+1)^{p^+_\Psi-1}}{2^{kN}}  \int_{A_{k-1}\cup A_k}\int_{A_{k-1}\cup A_k}\Psi\left(|x|^{\alpha_1}|y|^{\alpha_2}|D_su(x,y)|\right)d\mu,\;\;k\ge 0,
\end{multline*}
for some $C_3>0.$
This yields
\begin{multline*}
\frac{2^{kN}}{(k+1)^{p^+_\Psi-1}}\Psi\left(2^{ -k\gamma}|(u)_{A_k}|\right)\leq \frac{2^{(k-1)N}}{(k+1/2)^{p^+_\Psi-1}}\Psi\left(2^{-(k-1)\gamma}|(u)_{A_{k-1}}|\right)\\+C_3\int_{A_{k-1}\cup A_k}\int_{A_{k-1}\cup A_k}\Psi\left(|x|^{\alpha_1}|y|^{\alpha_2}|D_su(x,y)|\right)d\mu.
\end{multline*}
We sum up the above inequalities to obtain
\begin{multline}\label{log-eqn5}
\sum_{k=0}^{n_0}\frac{2^{kN}}{(k+1)^{p^+_\Psi-1}}\Psi\left(2^{-k\gamma}|(u)_{A_k}|\right)\leq \sum_{k=0}^{n_0} \frac{2^{(k-1)N}}{(k+1/2)^{p^+_\Psi-1}}\Psi\left(2^{ -(k-1)\gamma}|(u)_{A_{k-1}}|\right)\\+C_3\int_{\mathbb{R}^N}\int_{\mathbb{R}^N}\Psi\left(|x|^{\alpha_1}|y|^{\alpha_2}|D_su(x,y)|\right)d\mu.
\end{multline}
Since $u\equiv 0$ on $A_{-1}$, we have $(u)_{A_{-1}}=0$. Thus, by re-indexing, we get
\begin{align*}
    \sum_{k=0}^{n_0}\frac{2^{(k-1)N}}{(k+1/2)^{p^+_\Psi-1}}\Psi\left(2^{-(k-1)\gamma}|(u)_{A_{k-1}}|\right)&=\sum_{k=0}^{n_0-1}\frac{2^{kN}}{(k+3/2)^{p^+_\Psi-1}}\Psi\left(2^{-k\gamma}|(u)_{A_k}|\right)\\&\leq \sum_{k=0}^{n_0}\frac{2^{kN}}{(k+3/2)^{p^+_\Psi-1}}\Psi\left(2^{-k\gamma}|(u)_{A_k}|\right).
\end{align*} 
Therefore,  from \eqref{log-eqn5}  we get 
\begin{multline*}
     \sum_{k=0}^{n_0}\left\{\frac{2^{kN}}{(k+1)^{p^+_\Psi-1}}-\frac{2^{kN}}{(k+3/2)^{p^+_\Psi-1}}\right\}\Psi\left(2^{-k\gamma}|(u)_{A_k}|\right)\\\leq C_3\int_{\mathbb{R}^N}\int_{\mathbb{R}^N}\Psi\left(|x|^{\alpha_1}|y|^{\alpha_2}|D_su(x,y)|\right)d\mu.
\end{multline*}
Now for each $k\in \mathbb{Z}$ with $k\ge0$, we have
\begin{equation*}
    \frac{1}{(k+1)^{p^+_\Psi-1}}-\frac{1}{(k+3/2)^{p^+_\Psi-1}}\asymp  \frac{1}{(k+1)^{p^+_\Psi}}.
\end{equation*}
 Thus, from  \eqref{log-eq3} we obtain
\begin{align*}
    \int_{\mathbb{R}^N}\frac{\Psi\left(|x|^{-\gamma}|u(x)|\right)}{\big(\log(2|x|/R)\big)^{p^+_\Psi}}dx\leq C_4\int_{\mathbb{R}^N}\int_{\mathbb{R}^N}\Psi\left(|x|^{\alpha_1}|y|^{\alpha_2}|D_su(x,y)|\right)d\mu,
\end{align*}
where $C_4$ is a positive constant independent of $u$. Hence the result follows from the above inequality by using $\Phi\asymp\Psi$. \qed
\smallskip

   \section{Weighted Orlicz-Hardy inequalities in local case} 
In this section, we establish the analogue of Theorem \ref{Orlicz-Hardy}, Theorem \ref{Orlicz-Hardy2}, and Theorem \ref{log-orlicz} for $s=1$. Recall that the proof of key lemma (Lemma \ref{Lemma-1-2}) that we used for proving these theorems required the fractional Orlicz-Poincar\'e-Wirtinger inequality (see Proposition \ref{lem-orlicz}). Thus, first, we establish a local version of 
Orlicz-Poincar\'e-Wirtinger inequality. Our proof follows the same lines as in the proof of classical Poincar\'e-Wirtinger inequality obtained by Evans (\cite[Theorem 1, Page 275]{Evans1998}).

\begin{prop}(\textbf{Orlicz-Poincar\'e-Wirtinger inequality}): Let $\Omega$ be a bounded Lipschitz domain in $\mathbb{R}^N,\,N\ge 1$. Let $\Omega_\lambda=\{\lambda x:x\in \Omega\},\,\lambda>0$. Then, for an Orlicz function  $\Phi$, there exists $C=C(\Phi,\Omega)>0$  so that
\begin{equation}\label{poincare2}
     \int_{\Omega_\lambda}\Phi(|u(x)-(u)_{\Omega_\lambda}|)dx\leq C \int_{\Omega_\lambda}\Phi\left(\lambda|\nabla u(x)|\right)dx,\;\;\; \forall\, u\in W^{1,\Phi}(\Omega_\lambda).
\end{equation}
\end{prop}

\begin{proof}
By changing the variable $y= \lambda x$ we can see that \eqref{poincare2} holds if and only if 
\begin{equation}\label{poincare3}
     \int_{\Omega}\Phi(|u(x)-(u)_{\Omega}|)dx\leq C \int_{\Omega}\Phi\left(|\nabla u(x)|\right)dx,\;\;\; \forall\, u\in W^{1,\Phi}(\Omega).
\end{equation}
Thus, it is enough to prove \eqref{poincare3}. We use the method of contradiction. Assume that \eqref{poincare3} does not hold. Then for each $n\in \mathbb{N}$, there exists   $u_n\in W^{1,\Phi}(\Omega)$ satisfying 
\begin{equation*}
     \int_{\Omega}\Phi(|u_n(x)-(u_n)_\Omega|)dx\geq n \int_{\Omega}\Phi\left(|\nabla u_n(x)|\right)dx.
\end{equation*}
Since $n\Phi(t)\geq \Phi(n^{1/p^+_\Phi}t)$ for every $t\in [0,\infty)$ (see \eqref{H3}), the definition of the Luxemburg norm gives
\begin{equation}\label{eq1-poin}
    \|u_n-(u_n)_\Omega\|_{L^\Phi(\Omega)}\geq n^{1/p^+_\Phi}\|\nabla u_n\|_{L^\Phi(\Omega)}.
\end{equation}
Set $$v_n:=\frac{u_n-(u_n)_\Omega}{\|u_n-(u_n)_\Omega\|_{L^\Phi(\Omega)}},\;\;\;n\in \mathbb{N}.$$ Then, one can verify that 
\begin{equation}\label{poi-eqn2}
    (v_n)_\Omega=0,\;\;\;\|v_n\|_{L^\Phi(\Omega)}=1,\;\;\; n\in \mathbb{N}.
\end{equation}
Therefore, \eqref{eq1-poin} gives
\begin{equation}\label{poi-eq3}
    \|\nabla v_n\|_{L^\Phi(\Omega)}\leq \frac{1}{n^{1/p^+_\Phi}}.
\end{equation}
Thus, $(v_n)$ is a bounded sequence in $W^{1,\Phi}(\Omega)$, and hence by the compactness of the embedding $W^{1,\Phi}(\Omega)\hookrightarrow L^\Phi(\Omega)$ (see \cite[Theorem 8.32, Theorem 3.35]{Adams1975}), we get a sub-sequence, for the simplicity we denoted by $(v_{n})$ itself, and  $v\in L^\Phi(\Omega)$ such that
$v_{n}\to v\,\text{in}\,L^\Phi(\Omega).$ Since 
$\Omega$ is bounded, we also have  $v_{n}\to v$ in $L^1(\Omega)$.
Therefore, from \eqref{poi-eqn2} it follows that 
\begin{equation}\label{poi-eq4}
   (v)_\Omega=0,\;\;\; \|v\|_{L^\Phi(\Omega)}=1.
\end{equation}
  Next we show that  $\frac{\partial v}{dx_i}=0$ (the distributional derivative), for $i\in \{1,\dots,N\}$. For
 $w\in \mathcal{C}_c^\infty(\Omega)$, we have
\begin{equation*}
 \left\langle \frac{\partial v}{dx_i}, w\right\rangle=-   \int_\Omega v\frac{\partial w}{dx_i}dx=-\lim_{n\to \infty}\int_\Omega v_n\frac{\partial w}{dx_i}dx=\lim_{n\to \infty}\int_\Omega \frac{\partial v_n}{dx_i}wdx.
\end{equation*}
Further, the H\"older inequality for the Orlicz function (see \cite[Page 234]{Adams1975}) gives
$$\left|\int_\Omega \frac{\partial v_n}{dx_i}wdx\right|\leq \int_\Omega |\nabla v_n||w|dx\le 2\|\nabla v_n\|_{L^\Phi(\Omega)}\|w\|_{L^{\tilde{\Phi}}(\Omega)},$$ where $\tilde{\Phi}$ is the complementary Orlicz function to $\Phi.$ 
Consequently, by \eqref{poi-eq3} we get 
$$\left\langle \frac{\partial v}{dx_i}, w\right\rangle=0,\;\;\;\forall\,w\in \mathcal{C}_c^\infty(\Omega),\;\;i\in\{1,\dots,N\}.$$ Therefore,
$v\in W^{1,\Phi}(\Omega)$ with $\nabla v=0$ a.e. in $\Omega$. By the connectedness of $\Omega$,  $u$ must be a constant in $\Omega$. A contradiction to \eqref{poi-eq4}. 
\end{proof}

\smallskip

Now, we state the local analogue of Theorem \ref{Orlicz-Hardy}, Theorem \ref{Orlicz-Hardy2}, and Theorem \ref{log-orlicz}. The proof follows directly from the approaches used in the proofs of Theorem \ref{Orlicz-Hardy}, Theorem \ref{Orlicz-Hardy2}, and Theorem \ref{log-orlicz}, respectively. So, we omit the proof here.
\begin{thm}\label{Orlicz-Hardy*} Let $N\geq 1,\,R>0$, $\alpha\in \mathbb{R}$,  and let $\gamma:=1-\alpha$. For an Orlicz function $\Phi,$  
 \begin{enumerate}[(i)]
\item  if $\gamma<N/p^\oplus_\Phi$, then
 \begin{equation}\label{Hardy*}
    \int_{\mathbb{R}^N}\Phi\left(\frac{|u(x)|}{|x|^\gamma}\right)dx\leq C\int_{\mathbb{R}^N}\Phi\left(|x|^{\alpha}|\nabla u(x)|\right)dx,\;\;\;\forall\,u\in \mathcal{C}^1_c(\mathbb{R}^N),
\end{equation}

\item if $\gamma>N/p^\ominus_\Phi$, then
  \begin{equation*}
    \int_{\mathbb{R}^N}\Phi\left(\frac{|u(x)|}{|x|^\gamma}\right)dx\leq C\int_{\mathbb{R}^N}\Phi\left(|x|^{\alpha}|\nabla u(x)|\right)dx,\;\;\;\forall\, u\in \mathcal{C}^1_c(\mathbb{R}^N\setminus\{0\}),
\end{equation*}

\item if $p^\oplus_\Phi$ is attained  and $\gamma=N/p^\oplus_\Phi$, then for every $u\in \mathcal{C}_c^1(\mathbb{R}^N)$ with $supp(u)\subset B_R(0),$  
 \begin{equation*}
           \int_{B_R(0)}\frac{\Phi(|x|^{-\gamma}|u|)}{\big(\log(2R/|x|)\big)^{p^\oplus_\Phi}}dx\leq C\int_{B_R(0)}\Phi\left(|x|^{\alpha}|\nabla u(x)|\right)dx,
        \end{equation*}

\item if $p^\ominus_\Phi$ is attained at $\Psi\asymp\Phi$ and $\gamma=N/p^\ominus_\Phi$, then for every $u\in \mathcal{C}_c^1(\mathbb{R}^N)$ with $supp(u)\subset B_R(0)^c,$
        \begin{equation*}
           \int_{B_R(0)^c}\frac{\Phi(|x|^{-\gamma}|u|)}{\big(\log(2|x|/R)\big)^{p^+_\Psi}}dx\leq C\int_{B_R(0)^c}\Phi\left(|x|^{\alpha}|\nabla u(x)|\right)dx,
        \end{equation*}
        \end{enumerate}
      where  $C$ is a positive constant
independent of $u$.
If $\gamma> N/p^\ominus_\Phi$, then
        \eqref{Hardy*} fails.  Furthermore, if $p^\ominus_\Phi$ is attained, then \eqref{Hardy*}
fails also for $\gamma= N/p^\ominus_\Phi$.

\end{thm}

\section{Concluding remarks and examples:}
In this section, we discuss some remarks related to Theorem \ref{Orlicz-Hardy}, Theorem \ref{Orlicz-Hardy2}, and Theorem \ref{log-orlicz}. We also provide the values of $p^-_\Phi$, $p^+_\Phi,$ $p^\ominus_\Phi,$ and $p^\oplus_\Phi$ for some Orlicz functions. For an Orlicz function, the following pictures summarise the values of $\gamma$ for which we know weighted fractional Orlicz-Hardy inequality holds or fails on $\mathbb{R}^N$ and $\mathbb{R}^N\setminus \{0\}$, respectively. 

\vspace{1cm}
\begin{center}
\begin{tikzpicture}[scale=7]
\draw[<-, thick] (-0.1,0) -- (1.5,0);
\draw[->, thick] (-0.1,0) -- (1.7,0);
\foreach \x/\xtext in {0.5/$\frac{N}{p^\oplus_\Phi}$,1.11/$\frac{N}{p^\ominus_\Phi}$}
    \draw[thick] (\x,0.5pt) -- (\x,-0.5pt) node[below] {\xtext};
\draw (0.16,0.5pt) node[above] { $\gamma<\frac{N}{p^\oplus_\Phi}$};
\draw (1.4,0.5pt) node[above] {  $\gamma>\frac{N}{p^\ominus_\Phi}$};
\draw (0.8,0.5pt) node[above] {grey area};
\draw (0.2,-0.5pt) node[below] {holds};
\draw (1.4,-0.5pt) node[below] {fails};
\draw (0.8,-5pt) node[below] {Weighted fractional Orlicz-Hardy inequality on $\mathbb{R}^N$ };
 \draw[-{Arc Barb[]}, ultra thick, green!70!black] (-0.09,.0) -- (.5,0.0);
  \draw[ ultra thick, gray!70!black] (.5,.0) -- (1.101,0.0);
\draw[-{Arc Barb[]}, ultra thick,
red!70!black] (1.69,0) -- (1.104,0.0);
\end{tikzpicture}
\end{center}
\vspace{1.5cm}
\begin{center}
\begin{tikzpicture}[scale=7]
\draw[<-, thick] (-0.1,0) -- (1.5,0);
\draw[->, thick] (-0.1,0) -- (1.7,0);
\foreach \x/\xtext in {0.5/$\frac{N}{p^\oplus_\Phi}$,1.105/$\frac{N}{p^\ominus_\Phi}$}
    \draw[thick] (\x,0.5pt) -- (\x,-0.5pt) node[below] {\xtext};
\draw (0.16,0.5pt) node[above] {$\gamma<\frac{N}{p^\oplus_\Phi}$};
\draw (0.8,0.5pt) node[above] {grey area};
\draw (1.4,0.5pt) node[above] {  $\gamma>\frac{N}{p^\ominus_\Phi}$};
\draw (0.2,-0.5pt) node[below] {holds};
\draw (1.4,-0.5pt) node[below] {holds};
\draw (0.8,-5pt) node[below] {Weighted fractional Orlicz-Hardy inequality on $\mathbb{R}^N\setminus\{0\}$};
  \draw[-{Arc Barb[]}, ultra thick, green!70!black] (-0.09,.0) -- (.5,0.0);
  \draw[ ultra thick, gray!70!black] (.5,.0) -- (1.101,0.0);
\draw[-{Arc Barb[]}, ultra thick,
green!70!black] (1.69,0) -- (1.104,0.0);
\end{tikzpicture}
    
\end{center}

\begin{remark}\label{rem1} For any Orlicz function $\Phi$, we have the complete knowledge of \eqref{Hardy} and \eqref{Hardy-0} except for $\gamma\in \left[N/p^\oplus_\Phi,N/p^\ominus_\Phi\right]$ (see Theorem \ref{Orlicz-Hardy} and Theorem \ref{Orlicz-Hardy2}). Hence we call the interval $\left[N/p^\oplus_\Phi,N/p^\ominus_\Phi\right]$ as a \textit{\textbf{grey area}}  of $\Phi$.
 Outside this grey area,   \eqref{Hardy} holds for any $\gamma \in \left(-\infty,N/p^\oplus_\Phi\right)$, and fails for any $\gamma\in \left(N/p^\ominus_\Phi,\infty\right)$ (see Theorem \ref{Orlicz-Hardy}). On the other hand, \eqref{Hardy-0} holds for any $\gamma\in  \left[N/p^\oplus_\Phi,N/p^\ominus_\Phi\right]^c$ (see Theorem \ref{Orlicz-Hardy2}). If $\Phi\asymp A_p,$ then we have $p_\Phi^\ominus = p_\Phi^\oplus=p.$ Thus the grey area of such $\Phi$ reduces to a singleton set  $\{N/p\}.$
 \end{remark}
 
  \begin{remark}\label{rem2}
      For the Orlicz function $S(t)=t^p+t^q$ with $q> p$,  the grey area is $\left[N/q,N/p\right]$ (see Example \ref{exm-o}). However, \eqref{Hardy} fails for any $\gamma$ lies in the grey area of $S$, whereas \eqref{Hardy-0} holds also for $\gamma$ in $ \left(N/q,N/p\right)$. We justify these facts as below:
      \begin{enumerate}[(i)]
          \item For $S,$ $p^\ominus_S=p,\;p^\oplus_S=q$ and
\begin{equation*}
    \int_{\mathbb{R}^N}S\left(\frac{|u(x)|}{|x|^\gamma}\right)dx\geq  \int_{B_1(0)}\frac{|u(x)|^q}{|x|^{\gamma q}}dx,\;\;\;\forall\,u\in \mathcal{C}_c^1(\mathbb{R}^N),
\end{equation*}
 where $B_1(0)$ is the open ball centred at the origin with radius $1$. For  $\gamma\geq N/q,$  
 $|x|^{-\gamma q}$ is not locally integrable on $\mathbb{R}^N$ and  hence \eqref{Hardy} fails even for $\gamma \in \left[N/q,\infty\right).$
  Thus by Theorem \ref{Orlicz-Hardy},  \eqref{Hardy} holds for $\gamma<N/q$ and fails for $\gamma\ge N/q.$
  
\smallskip

\item  Let $\gamma\in \left(N/q,N/p\right).$ Then, we apply 
 Theorem \ref{Orlicz-Hardy} and Theorem \ref{Orlicz-Hardy2} to the Orlicz functions $ t^p$ and $ t^q$, respectively to obtain
    \begin{equation*}
         \int_{\mathbb{R}^N}\left(\frac{|u(x)|}{|x|^{\gamma }}\right)^pdx\leq C \int_{\mathbb{R}^N}\int_{\mathbb{R}^N}\left(|x|^{\alpha_1 }|y|^{\alpha_2 }|D_su(x,y)|\right)^pd\mu,\;\;\;\forall\,u\in\mathcal{C}_c^1(\mathbb{R}^N\setminus\{0\}),
    \end{equation*}
     \begin{equation*}
\int_{\mathbb{R}^N}\left(\frac{|u(x)|}{|x|^{\gamma }}\right)^qdx\leq C \int_{\mathbb{R}^N}\int_{\mathbb{R}^N}\left(|x|^{\alpha_1 }|y|^{\alpha_2 }|D_su(x,y)|\right)^qd\mu,\;\;\;\forall\,u\in\mathcal{C}_c^1(\mathbb{R}^N\setminus\{0\}).
    \end{equation*}   
By adding the above two inequalities we concludes that \eqref{Hardy-0} holds for $\gamma\in \left(N/q,N/p\right).$
      \end{enumerate}    
   \end{remark}

\smallskip

 If $\alpha_1=\alpha_2=0$, then  $\gamma=s-\alpha_1-\alpha_2=s$ and \eqref{Hardy-0} reduce to 
 \begin{equation}\label{Hardy-2}
    \int_{\mathbb{R}^N}\Phi\left(\frac{|u(x)|}{|x|^\gamma}\right)dx\leq C\int_{\mathbb{R}^N}\int_{\mathbb{R}^N}\Phi\left(|D_su(x,y)|\right)d\mu,\;\;\;\forall\,u\in \mathcal{C}_c^1(\mathbb{R}^N\setminus\{0\}).
 \end{equation} 
 In \cite[Page 578]{Dyda2007}, authors proved that for $\Phi(t)=t^p$ and $\gamma=N/p,$ the above inequality fails to hold. We prove this result for $S$.

\begin{lemma}\label{A} For $S(t)=t^p+t^q$ with $q> p,$ \eqref{Hardy-2} fails for $\gamma=N/q.$ 
\end{lemma}
\begin{proof}
Let $f(t)=e^{-\left(\frac{2-t}{t-1}\right)^2},\; t\in (1,2]$ and for large $n\in \mathbb{N}$
define
        \[g_n(t)=
     \begin{cases}
     \text{$0$} &\quad\text{$0\le t\le \frac{1}{n},\quad\text{$t\ge 2,$}$}\\
     f(nt) & \quad\text{$\frac{1}{n}< t\le \frac{2}{n}$},\\ \text{$1$} &\quad\text{$\frac{2}{n}\le t\le 1$,} \\
     f(3-t) & \quad\text{$1\le t< 2$}.
     \end{cases}
\]
Set $u_n(x)=g_n(|x|), \,x\in \mathbb{R}^N.$ Then one can verify that $u_n\in \mathcal{C}_c^1(\mathbb{R}^N\setminus\{0\})$, $0\leq u_n\le 1,$ and  
   \[   |\nabla u_n(x)|\le 
     \begin{cases} 
       Cn &\quad \text{ $\frac{1}{n}<|x|<\frac{2}{n}$,}\\
       \text{$C$} & \quad\text{$|x|>1$},
     \end{cases}
\] 
for some $C>0.$
We denote  $A(r_1,r_2)=B_{r_2}(0)\setminus B_{r_1}(0).$  Using the symmetry  of the integrand, we can see that
 \begin{multline*}
     \int_{\mathbb{R}^N\setminus\{0\}}\int_{\mathbb{R}^N\setminus\{0\}}S(|D_su_n(x,y)|)d\mu\\\leq 2 \left(\int_{B_{{2/n}}(0)}\int_{B_{3/n}(0)}+\int_{B_{{2/n}}(0)}\int_{A(3/n,\infty)}\right)S(|D_su_n(x,y)|)d\mu\\+2\left(\int_{A(2/n,2)}\int_{A(1,3)}+\int_{A(2/n,2)}\int_{A(3,\infty)}\right)S(|D_su_n(x,y)|)d\mu\\=:2(J_1^r+J_2^r+J_3^r+J_4^r).
 \end{multline*}
For $J_1$, using $q\ge p$ and $\gamma=s=N/q,$ we get 
\begin{multline*}
     J_1= \int_{B_{{2/n}}(0)}\int_{B_{3/n}(0)}S\left(\frac{|u_n(x)-u_n(y)|}{|x-y|^{s}}\right)d\mu\leq  \int_{B_{{2/n}}(0)}\int_{B_{3/n}(0)}S\left(\frac{Cn|x-y|}{|x-y|^{s}}\right)\frac{dxdy}{|x-y|^N}\\=  \int_{B_{{2/n}}(0)}\int_{B_{3/n}(y)}S\left(\frac{Cn|z|}{|z|^{s}}\right)\frac{dzdy}{|z|^N}\le \int_{B_{{2/n}}(0)}\int_{B_{6/n}(0)}\left(\frac{Cn^p}{|z|^{(s-1)p+N}}+\frac{Cn^q}{|z|^{(s-1)q+N}}\right)dzdy
      \\\le \frac{C}{n^{N-sp}}+\frac{C}{n^{N-sq}} \le  C.
\end{multline*}
For $J_2$, using $q\ge p$ and $\gamma=s=N/q,$ we get
\begin{multline*}
     J_2\leq  \int_{B_{{2/n}}(0)}\int_{A(3/n,\infty)}S\left(\frac{2}{|x-y|^{s}}\right)\frac{dxdy}{|x-y|^N}\leq\int_{B_{{2/n}}(0)}\int_{|z|\ge 1/n} S\left(\frac{2}{|z|^{s}}\right)\frac{dzdy}{|z|^N}\\= \int_{B_{{2/n}}(0)}\int_{|z|\ge 1/n} \left(\frac{2^p}{|z|^{sp+N}}+\frac{2^q}{|z|^{sq+N}}\right)dzdy \leq \frac{C}{n^{N-sp}}+\frac{C}{n^{N-sq}}\le C.
\end{multline*}
For $J_3$, we get
\begin{multline*}
    J_3\leq  \int_{A(2/n,2)}\int_{A(1,3)}S\left(C\frac{|x-y|}{|x-y|^{s}}\right)\frac{dxdy}{|x-y|^N}\leq  \int_{B_2(0)}\int_{B_6(0)}S\left(C|z|^{1-s}\right)\frac{dzdy}{|z|^N}\\\leq C\int_{B_2(0)}\int_{B_6(0)}\left(|z|^{(1-s)p-N}+|z|^{(1-s)q-N}\right)dzdy\le C.
\end{multline*}
For $J_4$, we get
\begin{multline*}
     J_4\leq  \int_{A(2/n,2)}\int_{A(3,\infty)}S\left(\frac{2}{|x-y|^{s}}\right)\frac{dxdy}{|x-y|^N} \leq  \int_{B_2(0)}\int_{|z|\ge 1}S\left(\frac{2}{|z|^{s}}\right)\frac{dzdy}{|z|^N}\\\leq C \int_{|z|\ge 1}\left(|z|^{-sp-N}+|z|^{-sq-N}\right)dz\le C.
\end{multline*}  
Thus, combining 
$J_1,\,J_2,\,J_3,\,J_4,$ we get 
\begin{equation}\label{con}
    \int_{\mathbb{R}^N\setminus\{0\}}\int_{\mathbb{R}^N\setminus\{0\}}S\left(|D_su_n(x,y)|\right)d\mu<\infty.
\end{equation}
 Next, we use $\gamma =N/q$ to estimate the left hand side of \eqref{Hardy-2} as below:
$$\int_{\mathbb{R}^N\setminus\{0\}}S\left(\frac{|u_n(x)|}{|x|^\gamma}\right) dx\geq \int_{\{2/n\le |x|\le1\}}S\left(\frac{|u_n(x)|}{|x|^{\gamma}}\right)dx\ge \int_{\{2/n\le |x|\le1\}}\frac{1}{|x|^{\gamma q}}=\log \frac{n}{2}.$$
The above estimate together with \eqref{con} shows that \eqref{Hardy-2} fails to hold for $\gamma=N/q=N/p_S^\oplus.$
 \end{proof}
 
\smallskip

\noindent  \textbf{Open problems:} Can we extend  the above lemma for $\gamma=N/q$ with  $\alpha_1\ne 0$ or $\alpha_2\ne 0?.$ Is it true that \eqref{Hardy-0} holds for $\gamma=N/p$?
\smallskip 

\begin{remark}
 Since the Orlicz functions $S(t)=t^p+t^q$ and $M(t)=\max\{t^p,t^q\}$ are equivalent ($S\asymp M$), the above remark and lemma are also applicable for $M.$   
\end{remark}

\smallskip

Next, we compute  $p^-_\Phi,$ $p^+_\Phi,$ $p^\ominus_\Phi,$ and $p^\oplus_\Phi$  of some Orlicz function $\Phi$. We refer \cite{Mihai2012} and \cite[Example 2.4]{salort-jde} for the computation of   $p^-_\Phi$ and $p^+_\Phi$ of several Orlicz functions. 

\begin{example}\label{exm-o} In the following table, we list some Orlicz functions for which $p_\Phi^\ominus$ and $p_\Phi^\oplus$ are attained at $\Phi$ (i.e., $p_\Phi^\ominus=p^-_\Phi$ and $p_\Phi^\oplus=p^+_\Phi$).

\begin{center}
    \setlength{\arrayrulewidth}{0.5mm}
\setlength{\tabcolsep}{18pt}
\renewcommand{\arraystretch}{1.5}

\begin{tabular}{p{4cm}p{2cm}p{2cm}}
\hline

$\Phi(t)$ & $p^\ominus_\Phi=p^-_\Phi$  & $p^\oplus_\Phi=p^+_\Phi$  \\
\hline

 $t^p+t^q;\; q> p>1$  &$p$ & $q$
  \vspace{.3cm}\\ 
$\max\{t^p,t^q\};\; q> p>1$ &
 $p$ & $q$  \vspace{.3cm} \\

$t^p\ln (1+t);\;p\ge 1$ &  $p$ & $p+1$  \vspace{.3cm} \\

$(1+t)\ln(1+t)-t$  &$1$ & $2$  \vspace{.3cm} \\  

\hline
\end{tabular}
\vspace{.5cm}
\end{center}

\noindent To compute the above values, we set $f_\Phi(t)=\frac{t\varphi(t)}{\Phi(t)},\; t>0.$  Then
\begin{equation*}
p_\Phi^-=\inf_{t>0}f_\Phi(t),\qquad
p_\Phi^+=\sup_{t>0}f_\Phi(t).
 \end{equation*}
Recall that
$$
 p_\Phi^\ominus= \sup\{p_\Psi^-: \Phi \asymp \Psi\}, \qquad
      p_\Phi^\oplus= \inf\{p_\Psi^+: \Phi \asymp \Psi\}.      
  $$
If $\Phi\asymp\Psi,$
 then using \eqref{H3} we obtain   $\Psi(t)\leq \max\big\{t^{p^-_\Psi},t^{p^+_\Psi}\big\}\Psi(1)$ (see \eqref{H3}). Thus by the equivalence, we get $C=C(\Phi,\Psi)>0$ such that
\begin{equation}\label{b-phi}
    \frac{\Phi(t)}{t^{p^-_\Psi}\Psi(1)}\le C, \quad \forall\,t\in (0,1)\;\;\text{and}\quad \frac{\Phi(t)}{t^{p^+_\Psi}\Psi(1)}\le C,\;\;\; \forall\, t>1.
\end{equation}
The above inequalities help us to compute $p_\Phi^\ominus$ and $p_\Phi^\oplus$.

 \noindent$(i)$ For $S(t)=t^p+t^q$ with $q> p$, we have $$f_S(t)=(pt^p+qt^q)(t^p+t^q)^{-1},\quad t>0.$$ Thus, $p^-_S=\inf_{t>0}f_S(t)=p$ and $p_S^+=\sup_{t>0}f_S(t)=q$. If $\Psi\asymp S,$ then by \eqref{b-phi} we have
 $$  \frac{t^p+t^q}{t^{p^-_\Psi}}\le C, \quad \forall\,t\in (0,1)\;\;\text{and}\quad \frac{t^p+t^q}{t^{p^+_\Psi}}\le C,\;\;\; \forall\, t>1.$$ Therefore,
 we must have  $p\ge p_\Psi^-$  and $q\le p^+_\Psi$. Now, since $p^-_S=p$ and $p^+_S=q,$ we conclude that  $p^\ominus_S=p$ and $p^\oplus_S=q.$
    \smallskip
    
\noindent $(ii)$ For $M(t)=\max\{t^p,t^q\}$ with $q>p$, we have
    $$f_M(t)=p\chi_{(0,1)}(t)+q\chi_{[1,\infty)}(t),\quad t>0,$$ and hence $p^-_M=p$ and $p^+_M=q.$ Since  
 $S\asymp M$, we have $p^\ominus_M=p$ and $p^\oplus_M=q.$
    \smallskip
    
\noindent $(iii)$ For $\Phi_1(t)=t^p\ln(1+t)$ with $p\ge 1,$ we have
    $$f_{\Phi_1}(t)=p+\frac{t}{(t+1)\ln (1+t)},\quad t>0.$$ One can compute that $p_{\Phi_1}^-=p$ and $p_{\Phi_1}^+=p+1$. If $\Psi\asymp \Phi_1,$ then by \eqref{b-phi} we have
 $$  \frac{t^p\ln(1+t)}{t^{p^-_\Psi}}\le C, \quad \forall\,t\in (0,1)\;\;\text{and}\quad \frac{t^p\ln(1+t)}{t^{p^+_\Psi}}\le C,\;\;\; \forall\, t>1.$$
Therefore, we must have $p_\Psi^-\le p$ and $p_\Psi^+\ge p$. Consequently, $p^\ominus_{\Phi_1}=p$ and  $p^\oplus_{\Phi_1}\in [p,p+1].$
Next, we calculate the exact value of  $p^\oplus_\Phi$. If 
$\Psi\asymp \Phi_1$, using \eqref{H3} we get $C>0$ such that 
\begin{equation}\label{eq-last}
    \Phi_1(ab)\le Ca^{p^+_\Psi}\Phi_1(b),\quad\forall\, b>0,a>1.
\end{equation}
For $a>1,$ one can verify that  $\frac{\Phi_1(ab)}{\Phi_1(b)}$ is a decreasing function of $b$, and hence  $$\sup_{b>0}\frac{\Phi_1(ab)}{\Phi_1(b)}=a^p\lim_{b\to 0}\frac{\ln(1+ab)}{\ln(1+b)}=a^{p+1},\quad a>1.$$
Thus, from \eqref{eq-last} we get $a^{p+1}\le Ca^{p^+_\Psi},\quad\forall\,a>1.$
Therefore, we must have $p^+_\Psi\ge p+1$, and hence  $p_{\Phi_1}^\oplus=p+1$.
    \smallskip

\noindent $(iv)$ For $\Phi_2(t)=(1+t)\ln(1+t)-t,$ we have
    $$f_{\Phi_2}(t)=1+\frac{t-\ln(1+t)}{(1+t)\ln (1+t)-t},\quad t>0.$$
One can verify that, $p_{\Phi_2}^-=1$ and $p_{\Phi_2}^+=2$. Since $\varphi_2(t)=\ln(1+t),$ we have $\Phi_2(t)\asymp t\ln(1+t) $ (see \eqref{H1}). Now from $(iii)$ we get  $p_{\Phi_2}^\ominus=1$ and $p_{\Phi_2}^\oplus=2$.

\smallskip

 \smallskip

\end{example}

Next, we provide an example of an Orlicz function $\Phi$ such that $p^\ominus_\Phi\ne p_\Phi^-$ and $p^\oplus_\Phi\ne p_\Phi^+.$ 
   \begin{example}\label{ex-2}
Consider the Orlicz function
\[\Phi(t)=
       \begin{cases}
     \text{$t^2$} &\quad\text{$0\le t\le 1 $,}\\
     \text{$\frac{1}{9}\left(2t^3+12t-5\right)$} & \quad\text{$1\le t\le 2 $},\\ \text{$t^2-\frac{1}{9}$} &\quad\text{$2\le t< \infty $.}
     \end{cases}
  \]
     \smallskip
     
     \noindent
It is easy to verify that $\Phi\asymp A_2$. 
Let $f(t)=\frac{t\Phi^\prime(t)}{\Phi(t)},\;t>0$. Then
\begin{figure}[ht!]
\begin{minipage}{0.45\textwidth}
\epsfig{file=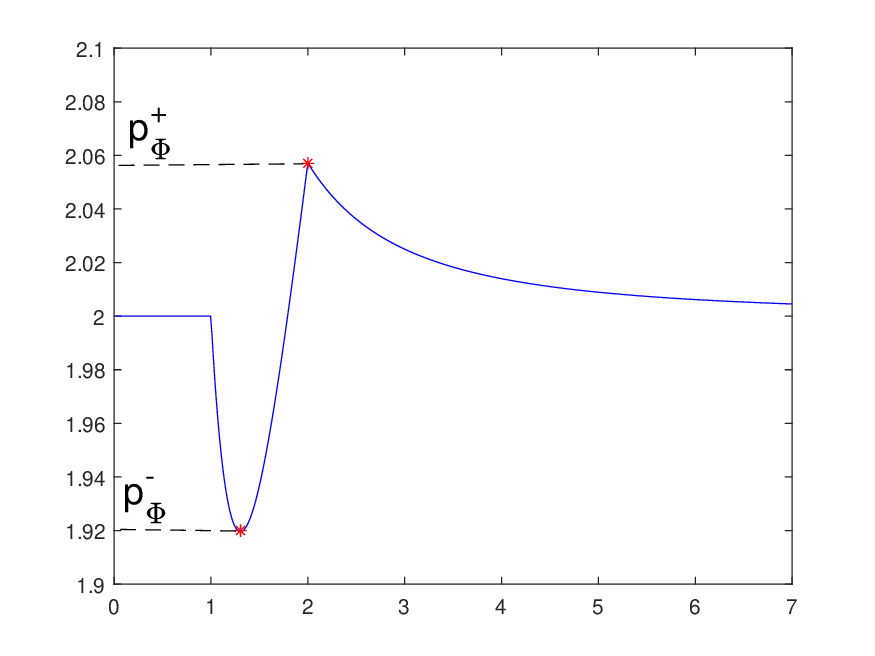,scale=0.45}\\
\centering{ Graph of $f$}
\end{minipage}\hspace{1cm}
\label{fig2}
\end{figure} 

\[f(t)=
     \begin{cases}
     \text{$2$} &\quad\text{$0< t\le 1$,}\\
     \text{$\frac{6(t^3+2t)}{2t^3+12t-5}$} & \quad\text{$1\le t\le 2,$}\\
     
     \text{$\frac{18t^2}{9t^2-1}$} &\quad\text{$2\le t< \infty.$}
     \end{cases}
\]
 By analysing the function  $f$, one can deduce the following:
\begin{align*}
p_\Phi^-<2 ,\quad  & p_\Phi^+=f(2)=\frac{72}{35}>2.
\end{align*}
Since $\Phi\asymp A_2,$ we have $p^\ominus_\Phi=p_\Phi^\oplus=2.$ 
\end{example}

\noindent{\textbf{Open problem:}}
     We anticipate that,  for each Orlicz function $\Phi$, there exists $\gamma_{\Phi}^*\in  \left[N/p^\oplus_\Phi,N/p_\Phi^\ominus\right]$ such that \eqref{Hardy} holds for $\gamma<\gamma_{\Phi}^*$ and \eqref{Hardy} fails for $\gamma\ge \gamma_{\Phi}^*.$ Similarly, \eqref{Hardy-0} holds for every $\gamma\neq \gamma_{\Phi}^* $
and fails for $\gamma=\gamma_{\Phi}^*.$ For $\Phi=A_p,$ we can see that $\gamma_{\Phi}^*=N/p.$

 \begin{center}
	{\bf Acknowledgments} 
 \end{center}
 P. Roy  is supported by SERB's Core Research Grant (CRG/2022/0078ERB. P. Roy would like to thank The Institute of Mathematical Sciences (IMSc), Chennai, India, where parts of the work are carried out.

\smallskip
\noindent{\bf Ethics \& disclosures}: On behalf of all authors, the corresponding author states that there is no conflict of interest.

\noindent{\bf Data availability}: 
We do not analyse or generate any datasets, because our work proceeds within a theoretical and mathematical approach. One can obtain the relevant materials from the references below.

\bibliographystyle{abbrvurl}
\bibliography{Reference}

\end{document}